\documentclass[draft]{siamltex}
\usepackage{amsfonts}
\usepackage{amsmath}
\usepackage{amssymb}
\usepackage{array}
\usepackage{color}
\usepackage{eucal}
\usepackage{framed}
\usepackage{graphicx}
\usepackage{mathrsfs}
\usepackage{relsize}
\usepackage{tikz}
\numberwithin{equation}{section}
\numberwithin{table}{section}
\numberwithin{figure}{section}
\newcommand{\Oh}{{\mathcal{T}_h}}
\newcommand{\Eh}{\mathcal{F}_h}

\newcommand{\dK}{{\partial K}}

\newcommand{\vt}{\boldsymbol{{\mathrm{v}}}}
\newcommand{\eg}{\boldsymbol{\mathsf{e}}}

\newcommand{\Vh}{\boldsymbol{V}_h}
\newcommand{\Qh}{\mathring{{Q}_h}}
\newcommand{\Wh}{W_h}
\newcommand{\Mh}{M_h}
\newcommand{\VV}{{\boldsymbol V}}

\newcommand{\GG}{{\mathcal{G}}}

\newcommand{\WW}{W}
\newcommand{\Vtilde}{\widetilde\VV}
\newcommand{\Wtilde}{\widetilde W}
\newcommand{\Vperp}{\widetilde\VV{}^\perp}
\newcommand{\Wperp}{\widetilde W^\perp}
\newcommand{\Vbd}{\gamma\Vperp}
\newcommand{\Wbd}{\gamma\Wperp}

\newcommand{\vv}{\bld{v}}
\newcommand{\ww}{w}

\newcommand{\uhat}{\widehat{u}_h}
\newcommand{\qhat}{\widehat{\boldsymbol{q}}_h}
\newcommand{\what}{\widehat{w}_h}
\newcommand{\uuhat}{\widehat{u}}
\newcommand{\wwhat}{\widehat{w}}
\newcommand{\muuhat}{\bld{\widehat{u}}}

\newcommand{\Piw}{\Pi_W}

\newcommand{\eu}{e_{u}}

\newcommand{\egg}{\mathrm{e}_{L}}
\newcommand{\euu}{\bld{e}_{u}}
\newcommand{\euuhat}{\bld{e}_{\widehat u}}
\newcommand{\epp}{{e}_{p}}
\newcommand{\dgg}{\mathrm{\delta}_{L}}
\newcommand{\duu}{\bld{\delta}_{u}}
\newcommand{\duuhat}{\bld{\delta}_{\widehat u}}
\newcommand{\dpp}{{\delta}_{p}}
\newcommand{\ddl}{{\mathrm{d}}_{L}}
\newcommand{\ddu}{{\mathrm{d}}_{u}}
\newcommand{\dduhat}{{\mathrm{d}}_{\widehat u}}
\newcommand{\ddp}{{\mathrm{d}}_{p}}
\newcommand{\dbeta}{{\mathrm{d}}_{{\beta}}}
\newcommand{\ddw}{{\mathrm{d}}_{w}}
\newcommand{\ddwhat}{{\mathrm{d}}_{\widehat w}}

\newcommand{\mg}{\mathrm{G}}
\newcommand{\mr}{\mathrm{r}}

\newcommand{\mwhat}{\widehat{\bld{v}}}
\newcommand{\mwwhat}{\widehat{\bld{v}}}
\newcommand{\muhat}{\widehat{\bld{u}}_h}

\newcommand{\mbeta}{{\bld P}_{\!h}}
\newcommand{\VVD}{{\VV}^{\mathrm{D}}(K)}
\newcommand{\WWD}{{\WW}^{\mathrm{D}}(K)}
\newcommand{\MD}{{M}^{\mathrm{D}}(\dK)}
\newcommand{\Pigg}{P_\GG}

\newcommand{\Pimm}{P_{\bld{M}}}
\newcommand{\Piww}{\Pi_{\VV}}
\newcommand{\Piq}{P_{Q}}

\newcommand{\OO}{\mathcal{O}_h}

\newcommand{\bint}[2]{( #1\,,\,#2 )_{\Oh}}
\newcommand{\bintEh}[2]{\langle #1\,,\,#2 \rangle_{\partial{\Oh}}}
\newcommand{\bintK}[2]{\langle #1\,,\,#2 \rangle_{\partial{K}}}

\newcommand{\inp}[2]{(#1\,, \, #2)_{\Oh}}
\newcommand{\tr}{\mathrm{tr}}

\newcommand{\n}{\boldsymbol{n}}
\newcommand{\bld}[1]{\boldsymbol{#1}}
\newcommand{\bR}{\mathbb R}
\newcommand{\ave}[1]{\,\{#1\}}

\newcommand{\pol}{\EuScript{P}}
\newcommand{\bpol}{\boldsymbol{\pol}}
\newcommand{\qol}{\EuScript{Q}}
\newcommand{\bqol}{\boldsymbol{\qol}}

\newcommand{\mm}{M}

\newcommand{\divs}{{\nabla\cdot}}
\newcommand{\grads}{{\nabla}}
\newcommand{\divv}{{\bld{\nabla\cdot}}}
\newcommand{\gradv}{{\bld{\nabla}}}

\newcommand{\lmax}{{\lambda_{\mathrm{c}}^\mathrm{max}}}

\newcommand{\ml}{\mathrm{L}}

\newcommand{\mzhat}{\widehat{\bld z}_h}
\newcommand{\vertiii}[1]{{\left\vert\kern-0.35ex\left\vert\kern-0.35ex\left\vert #1 
    \right\vert\kern-0.35ex\right\vert\kern-0.35ex\right\vert}}
\definecolor{blue}{rgb}{0,0,0}

\definecolor{ber}{rgb}{0,0,0}
\newcommand\ber[1]{\textcolor{ber}{#1}}

\title{%Discrete $\bld{H}^1$-estimates by \mm-decompositions: \\
Discrete $H^1$-inequalities for spaces admitting  \mm-decompositions}
\author{
Bernardo Cockburn
        \thanks{School of Mathematics, University of Minnesota, Vincent Hall,
                Minneapolis, MN 55455, USA, email: {\tt cockburn@math.umn.edu}.
                Supported in part by the National Science Foundation
                (Grant DMS-1522657) and by the University of Minnesota
                Supercomputing Institute.}
\and
Guosheng Fu
        \thanks{
{        Division of Applied Mathematics, Brown University, 182 George St,
Providence RI 02912, USA, email: {\tt guosheng\_fu@brown.edu}.}}
\and
Weifeng Qiu
\thanks{Corresponding author. Department of Mathematics, City University of Hong Kong,
              83 Tat Chee Avenue, Kowloon, Hong Kong, China, email: {\tt weifeqiu@cityu.edu.hk}.
              The work of Weifeng Qiu was partially supported by a grant from the Research Grants 
              Council of the Hong Kong Special Administrative Region, China (Project No. CityU 11302014).}
}

\begin{document}

\maketitle

\begin{abstract}
We find new  discrete $H^{1}\!$- and Poincar\'e-Friedrichs inequalities by studying the invertibility of 
the DG approximation of the flux for local spaces admitting \mm-decompositions. We then show \ber{how} to use 
\ber{these inequalities} to define and analyze new, superconvergent HDG and mixed methods for which the stabilization 
function is defined in such a way that the approximations satisfy new $H^1$-stability results with which 
their error analysis is greatly simplified. We apply this approach  to define a wide class of
energy-bounded, superconvergent HDG and mixed methods 
for the incompressible Navier-Stokes equations defined on unstructured meshes using, in 2D, general polygonal elements and, in 3D, general, flat-faced tetrahedral, prismatic, pyramidal and hexahedral elements.
\end{abstract}



\thispagestyle{plain} \markboth{B. Cockburn, et.~al.}{
Discrete $H^1$-inequalitites and superconvergent HDG methods for Navier-Stokes}

\centerline{{\bf Version of \today}}

\section{Introduction}
\label{sec:intro}
\ber{In this paper, we obtain new discrete stability inequalities with which we 
carry out the first a priori error analysis of a
wide class of hybridizable discontinuous Galerkin (HDG) and mixed methods for the Navier-Stokes 
equations. The methods are defined on unstructured meshes using, in 2D, general polygonal 
elements and, in 3D, general, flat-faced tetrahedral, prismatic, pyramidal and hexahedral elements. They are a direct extension of the 
corresponding methods introduced for the Stokes flow in \cite{CockburnFuQiu16}. We prove optimal error estimates in all the unknowns as 
well as superconvergence results for the approximate velocity. By this, we mean that a new approximation for the velocity can be obtained in an 
elementwise manner which converges faster than the original velocity approximation.
}

\ber{The unifying feature of the above-mentioned class of methods is that they are defined by using 
the theory of \mm-decompositions. Using this theory, superconvergent HDG and mixed methods have 
been devised for diffusion \cite{CockburnFuSayas16,CockburnFuM2D,CockburnFuM3D},  for linear 
incompressible flow \cite{CockburnFuQiu16}, and for linear elasticity \cite{CockburnFuElas}. 
The theory of \mm-decompostions has also been used to obtain commuting de Rham sequences
\cite{CockburnFuCommuting}. Here, we use it to obtain the above-mentioned new discrete 
inequalities.}

\ber{To better explain our results, we introduce the HDG and mixed methods for steady-state diffusion
\[
 \mathrm{c}\boldsymbol{q} + \nabla u = 0, \quad
\nabla \cdot \boldsymbol{q} = f  \text{ in  }\Omega,
\quad
\text{ and }
\quad u = g \text{ on } \partial \Omega,
\]
and introduce the concept of an \mm-decomposition. We then describe the inequalities we want to obtain and, finally, describe how we are going to apply them to the analysis of HDG and mixed methods for the Navier-Stokes equations.}

\subsection*{HDG methods and \mm-decompositions} \ber{To define the HDG methods,
we follow \cite{CockburnGopalakrishnanLazarov09}. 
Thus, we take  the domain $\Omega\subset \bR^d$  to be a polygon if $d=2$ and 
a polyhedron if $d=3$. We triangulate it with a conforming  mesh $\Oh:=\{K\}$ made of  
shape-regular  polygonal/polyhedral elements $K$. 
We set $\partial \Oh:=\{\partial K: \, K\in \Oh\}$, and denote by $\Eh$ the set of faces $F$ 
of the elements $K \in \Oh$. We also denote by $\mathcal{F}(K)$ the set of faces $F$ of the element $K$.}

\ber{The HDG method seeks an approximation to $(u, \boldsymbol{q}, u|_{\Eh})$,
$(u_h, \boldsymbol{q}_h, \uhat)$, in the  {\color{blue}finite dimensional} space $W_h \times \Vh
 \times M_h$, where
%\begin{subequations}
%\label{spaces}
\begin{alignat*}{3}
\boldsymbol{V}_h:=&\;\{\boldsymbol{v}\in\boldsymbol{L}^2(\Oh):&&\;\boldsymbol{v}|_K\in\boldsymbol{V}(K),&&\; K\in\Oh\},
\\
W_h:=&\;\{w\in{L}^2(\Oh):&&\;w|_K\in W(K),&&\; K\in\Oh\},
\\
M_h:=&\;\{\wwhat\in{L}^2(\Eh):&&\;\wwhat|_F\in M(F),&&\; F\in\Eh\},
\end{alignat*}
%\end{subequations}
and determines it as the only solution of the following weak formulation:
\begin{subequations}
\label{HDG equations}
 \begin{alignat}{3}
\label{HDG equations a}
 \bint{\mathrm{c}\,\boldsymbol{q}_h}{\boldsymbol{v}}&-\bint{u_h}{\nabla \cdot \boldsymbol{v}}   && + \bintEh{\uhat}{ \boldsymbol{v} \cdot \boldsymbol{n}} && = 0, \\
\label{HDG equations b}
& -\bint{\boldsymbol{q}_h}{\nabla w} && + \bintEh{\qhat \cdot \n}{w} && = \bint{f}{w}, \\
&\qhat\cdot \n = \boldsymbol{q}_h\cdot \n &&+ {\alpha (u_h - \uhat)} &&\quad\text{ on }\quad \partial \Oh,
\\
& && \quad \, \langle{\qhat \cdot \n},{\wwhat}\rangle_{\partial\Oh\setminus\partial\Omega} && = 0,\\
& && \quad \, \langle \widehat{u}_h,{\wwhat}\rangle_{\partial\Omega} && =  \langle {u_D},{\wwhat}\rangle_{\partial\Omega},
\end{alignat}
for all $(w, \boldsymbol{v}, \wwhat) \in W_h \times \Vh \times M_h$.
Here we write 
\end{subequations}
$\inp{\eta}{\zeta} := \sum_{K \in \Oh} (\eta, \zeta)_K,$ 
where $(\eta,\zeta)_D$ denotes the integral of $\eta\zeta$ over the domain $D \subset \mathbb{R}^n$. We also write
$\bintEh{\eta}{\zeta}:= \sum_{K \in \Oh} \langle \eta \,,\,\zeta \rangle_{\partial K},$
where $\langle \eta \,,\,\zeta \rangle_{D}$ denotes the integral of $\eta \zeta$ over the domain $D \subset \mathbb{R}^{n-1}$. %and where $\partial \Oh := \{ \partial K: K \subset \Oh \}$. 
When vector-valued functions are involved, we use a similar notation.}

\ber{The different HDG methods are obtained by choosing 
the local spaces $\boldsymbol{V}(K)$, $W(K)$ and 
\[
M(\partial K):=\{\wwhat\in L^2(\partial K): \;\wwhat|_F\in M(F)\mbox{ for all }F\in \mathcal{F}(K)\},
\]
and the {\em linear local stabilization} function ${\alpha}$. It turns out \cite{CockburnFuSayas16} that if we can decompose $\boldsymbol{V}(K)\times W(K)$ 
in such a way that
\begin{alignat*}{1}
\boldsymbol{V}(K)&=\widetilde{\boldsymbol{V}}(K) \oplus \widetilde{\boldsymbol{V}}^\perp(K),
\\
{W}(K)&=\widetilde{W}(K) \oplus \widetilde{W}^\perp(K),
\\
M(\partial K)&=\widetilde{\boldsymbol{V}}^\perp(K)\cdot\boldsymbol{n}|_{\partial K} \oplus  \widetilde{W}^\perp(K)|_{\partial K},
\end{alignat*}
and a couple of simple inclusion properties,
%where $\nabla\cdot\boldsymbol{V}(K)\times \nabla W(K)\subset\widetilde{W}(K)\times\widetilde{\boldsymbol{V}}(K)$and $M(\partial K)\supset {\boldsymbol{V}}(K)\cdot\boldsymbol{n}|_{\partial K}  +   {W}(K)|_{\partial K},$
that it is possible to find a stabilization function $\alpha$ such that the resulting HDG ($\alpha\neq0)$ 
or mixed method ($\alpha=0$) is superconvergent. Since this decomposition is essentially induced by the space $M(\partial K)$, it is called an $M(\partial K)$-decomposition of the space $\boldsymbol{V}(K)\times W(K)$. The explicit construction of those spaces for general polygonal 
elements was carried in \cite{CockburnFuM2D} (see the main examples in Table \ref{table:examples2D}) and for flat-faced general pyramids, prisms, and \ber{hexahedral} elements in 
\cite{CockburnFuM3D}. 
}
\subsection*{Invertibility of the discrete gradient operator}
In this paper, we study the invertibility properties of the mapping 
\begin{subequations}
\label{gc_hdg}
\begin{alignat}{3}
W(K)\times &M(\partial K) &&\longrightarrow && \;\bld{V}(K),
\\
 (u_h,&\,\uhat) &&\longmapsto && \;\bld{q}_h,
\end{alignat}
\vskip-.4truecm
\noindent where
\vskip-.6truecm
\begin{equation}
\label{gc_hdg_o}
(\mathrm{c}\,\bld{q}_h,\bld{v})_K=(u_h,\nabla\cdot\bld{v})_K-\langle\uhat, \bld{n}\cdot\bld{v}\rangle_{\partial K}\quad\forall\; \bld{v}\in \bld{V}(K),
\end{equation}
\end{subequations}
for spaces $\bld{V}(K)\times W(K)$ admitting an $M(\partial K)$-decomposition \cite{CockburnFuSayas16}. %Here, $(\eta,\zeta)_D$ denotes the integral of $\eta\zeta$ over the domain $D \subset \mathbb{R}^d$, and $\langle \eta \,,\,\zeta \rangle_{D}$ the integral of $\eta \zeta$ over the domain $D \subset \mathbb{R}^{d-1}$. 
This mapping is a discrete version of the 
{\em constitutive} equation relating a vector-valued function $\bld q$ and a scalar-valued function $u$:
\[
\mathrm{c}\,\bld{q} = - \nabla u,% \qquad \text{in $\Omega$,}
\]
where $\mathrm{c}$ and $\mathrm{c}^{-1}$ are bounded,  symmetric and uniformly positive definite matrix-valued functions, and
 has been used in, arguably, {\em all} DG and hybridized versions of mixed methods. \ber{In particular, it captures the first equation defining the HDG method for steady-state diffusion.}
We present new discrete versions of the estimates
\begin{alignat*}{3}
\|\nabla u\|^2_K&=\| \mathrm{c}\,\bld{q} \|^2 _K &&\quad\mbox{{\rm (trivial)}}, 
\\
 h_K^{-2}\| u-\overline{u}^K\|^2_K&\le C\, \| \mathrm{c}\,\bld{q} \|^2 _K&&\quad\mbox{{\rm (Poincar\'e-Friedrichs)}},
\end{alignat*}
where $\overline{\zeta}^{\,D}$ denotes the average of $\zeta$ \ber{on $D$} and $\|\cdot\|_D$ is the $L^2(D)$-norm.
They are expressed in terms of the (equivalent)
seminorms
\begin{subequations}
\label{localseminorms}
\begin{alignat}{3}
\label{localseminorm1}
| (u_h, \uhat)|^2_{1,K}:=& \|\nabla u_h\|^2_K+ h^{-1}_K\, \|u_h-\uhat\|^2_{\partial K}, 
\\
\label{localseminorm0}
| (u_h, \uhat)|^2_{\mbox{{\rm \tiny PF}},K}:=&\| u_h-\overline{\uhat}^{\;\partial K}\|^2_K+ h_K\, \|\uhat-\overline{\uhat}^{\;\partial K}\|^2_{\partial K},
\end{alignat}
\end{subequations}
and are, essentially, of the form
 \begin{alignat*}{2}
  |\,(u_h,\widehat{u}_h)\,|_{1,K}^2&\le C\,\left(\|\mathrm{c}\,\bld{q}_h\|_{K}^2 + 
%   \bintK{\alpha(u_h-\uhat)}{u_h-\uhat}
  {h_K^{-1}\|P_{M_S}( u_h-\uhat)\|_{\dK}^2}\right) &&\quad\mbox{{\rm ($H^1$)},}
  \\
 % \label{m-decomposition-dl2}
  h_K^{-2}\, |\,(u_h,\widehat{u}_h)\,|_{\mbox{{\rm \tiny PF}},K}^2&\le C\,\left( \|\mathrm{c}\,\bld{q}_h\|_{K}^2 + 
%   \bintK{\alpha(u_h-\uhat)}{u_h-\uhat}
  h_K^{-1}\|{P_{M_S}(u_h-\uhat)}\|_{\dK}^2\right)&&\quad\mbox{{\rm (Poincar\'e-Friedrichs)},}
 \end{alignat*}
where $M_S= M_S(\dK)$, 
referred to as the {\it stabilization} space, is an easy-to-compute subspace of the {space $M(\dK)$} whose dimension is chosen to be {\em minimal}, 
and $P_{M_S}$ is its corresponding $L^2$-projection.
%  ; typically, 
% it is the empty set or 
% an arbitrary face of the element $K$.
\ber{These inequalities, which are nothing but {\em stabilized} 
versions of $\inf$-$\sup$ conditions \cite{BoffiBrezziFortin13},
are the key ingredients for our analysis of HDG and mixed methods.
They generalize, to {\em all} spaces admitting \mm-decompositions, the $H^1$-inequality
obtained with {$M_{S}=\emptyset$} in \cite[Proposition 3.2]{EggerSchoberl10},
for the well known Raviart-Thomas spaces for simplexes, and, for smaller spaces, in \cite[Theorem 3.2]{ChungEngquist09} with
$M_S$ equal to the restriction of $M({\dK})$ onto an arbitrary face $F_K$
on which $\uhat$ was set to coincide with $u_h$.}

% We obtain the above-mentioned inequalities and use them to define \ber{and analyze} new superconvergent HDG and mixed methods. In particular, we consider the stability of their approximate solutions under perturbations which maintain the same Dirichlet boundary condition, and 
% show how to define the stabilization function of the method so that a new {\em stability} result of the approximations holds for the discrete, global $H^1$-like norm
% \[
% \vertiii{(u_h,\uhat)}_{1,\Oh}:=(\sum_{K\in\Oh} |(u_h,\uhat)|^2_{1,K})^{1/2},
% \]
% (this is a norm if $\uhat=0$ on $\partial \Omega$, which we can assume since the boundary conditions are held constant for the stability result)
% which greatly simplifies the error analysis of the methods.

\subsection*{Application to the Navier-Stokes equations} We show how to do that, not in the relatively simple case of convection-diffusion equations, but in the 
more difficult case of the velocity gradient-velocity-pressure formulation of the 
steady-state incompressible Navier-Stokes equations in two- and three-space dimensions:
\begin{subequations}
\label{ns-equation}
\begin{align}
\label{nsa}
\ml=\nabla \bld u&\;\qquad \text{in}\quad\Omega,\\
\label{nsb}
-\nu\divv \ml+\divs(\bld u\otimes\bld{u})+\grads p=\bld{f}&\;\qquad \text{in}\quad\Omega,\\
\label{nsc}
\divs\bld{u}=0&\;\qquad \text{in}\quad \Omega,\\
\label{nsd}
\bld{u}=\bld{0}&\;\qquad \text{on}\quad \partial\Omega,\\
\label{nse}
\int_\Omega p=0&\;,
\end{align}
\end{subequations}
where $\ml$ is the velocity gradient, $\bld{u}$ is the velocity, $p$ is the pressure, $\nu$ is the kinematic viscosity 
and $\bld{f}\in {L}^2(\Omega)^d$
is the external body force.

\ber{Let us compare our results with those in \cite{CesmeliogluCockburnQiu17} where the only error analysis for HDG methods for the Navier-Stokes equations has been recently carried out. Let   $(\bld{u}_h,\muhat)$ be an approximation of the velocity
$(\bld{u}|_\Omega, \bld{u}|_{\Eh})$, where $\Eh$ denotes the set of faces of the mesh $\Oh$ of the domain $\Omega$, and 
let $\mathrm{L}_h$ be an approximation of 
the velocity gradient $\mathrm{L}|_\Omega$. In \cite{CesmeliogluCockburnQiu17}, the authors considered unstructured meshes made of simplexes, spaces of polynomials of degree $k$, and a stabilization function
$\alpha$ such that
\[
\langle\alpha(\bld{u}_h-\muhat),\bld{u}_h-\muhat\rangle_{\dK} = h^{-1}_K \|(\bld{u}_h-\muhat)\cdot\bld{n}\|_{\dK}^2.
\]
%Note than the tangential component of the approximate velocity was {\em not} stabilized.
For this HDG method, optimal convergence order for all unknowns as well as the superconvergence of the velocity was obtained by using 
the novel upper bound
\[
\vertiii{(\bld{u}_h,\muhat)}_{1,\Oh}^2 \le C\,\sum_{K\in\Oh} ( \| \mathrm{L}_h\|^2_K + h^{-1}_K \|(\bld{u}_h-\muhat)\cdot\bld{n}\|_{\dK}^2),
\]
\ber{where the discrete $H^1$-norm $\vertiii{\cdot}_\Oh$ is given by
\[
\vertiii{(u_h,\uhat)}_{1,\Oh}:=(\sum_{K\in\Oh} |(u_h,\uhat)|^2_{1,K})^{1/2}.
\]}
In contrast, in this paper, \ber{stronger results} are obtained for a wide class of HDG and mixed methods defined on a variety of element shapes: general polgonal elements in 2D,  and tetrahedral, pyramidal, prismatic and hexahedral elements in 3D.} The local spaces defining these methods are those used for the corresponding methods for the Stokes equations of incompressible flow proposed in \cite{CockburnFuQiu16}; the stabilization function is not the same though. The spaces are constructed by using, as building blocks, the local spaces 
$\bld{V}(K)\times W(K)$ admitting an $M(\partial K)$-decomposition introduced in \cite{CockburnFuSayas16} for steady-state diffusion.

\ber{To obtain the new discrete inequalities, we proceed in two steps. First, we show that for all these methods, we have the
discrete $\bld{H}^1$-inequality
\[
\vertiii{(\bld{u}_h,\muhat)}_{1,\Oh}^2 \le C\,\sum_{K\in\Oh} ( \| \mathrm{L}_h\|^2_K + h^{-1}_K {\|P_{M_S}( \bld{u}_h-\muhat) \|_{\dK}^2}).
\]
We then show that if we {\em define} a stabilization function $\alpha$ such that
\[
\langle\alpha(\bld{u}_h-\muhat),\bld{u}_h-\muhat\rangle_{\dK} = {h^{-1}_K \|P_{M_S}(\bld{u}_h-\muhat) \|_{\dK}^2},
\]
we obtain new $\bld{H}^1$-boundedness results for the approximation, and new $\bld{H}^1$-stability inequalities, with which can easily obtain the above-mentioned convergence properties.}

\subsection*{Organization of the paper}The rest of the paper is organized as follows. In Section
\ref{sec:dh1}, {we present the general properties of the local spaces admitting \mm-decompositions and those of the the stabilization subspaces $M_S$;
specific choices of $M_S$ for the main spaces admitting \mm-decompositions are also provided. We then present and discuss our main result, namely, 
the new discrete inequalities of Theorem \ref{thm:dh1} which we prove} in Section  \ref{sec:proof-dh1}.
In Section \ref{sec:ns}, we define our HDG and mixed methods for the incompressible Navier-Stokes equations
and present their energy-boundedness and superconvergence properties; \ber{their proofs are provided} in Section \ref{sec:proof-ns}.
We end with some concluding remarks in Section \ref{sec:conclude}.

\section{The main result}
\label{sec:dh1}
In this Section, we present \ber{and discuss} our main result, namely, 
the discrete $\bld{H}^1$- and Poincar\'e-Friedrichs inequalities
of Theorem \ref{thm:dh1}\ber{; their proof is postponed to Section 3}. 
% \ber{We start with some notation.}
% We proceed as follows. We begin by introducing the notation. 
\ber{We first present the two ingredients 
needed to obtain these inequalities, namely, the spaces admitting \mm-decompositions and a 
stabilization subspace of the trace space $M(\partial K)$. 
% After \ber{ displaying and} discussing the inequalities, we show how to determine the stabilization
% \ber{function} of HDG methods for convection-diffusion methods by using the inequalities in question.
}

\subsection{Notation}

%When this approach is applied to convection-diffusion problems, the analysis of the resulting methods becomes much easier than in  \cite{ChenCockburn12}, since the auxiliary projection involving the convective derivative does not need to be used anymore.

\label{sec:notation}
Given a domain $D\subset \mathbb{R}^n$, we denote by
$\pol_k(D)$ and $\widetilde{\pol}_k(D)$ the space of polynomials of degree no
greater than $k$, and the space of homogeneous polynomials of degree $k$, respectively, defined on the domain $D$. 
When $D$ is a unit square with coordinates $(x,y)$, we denote by 
$\qol_k(D):= \pol_k(x)\otimes \pol_k(y)$ and $\widetilde{\qol}_k(D):=\widetilde\pol_k(x)\otimes \widetilde\pol_k(y)$ the space of 
tensor-product polynomials of degree no
greater than $k$, and the space of homogeneous tensor-product polynomials of degree $k$, respectively,
% , on the domain $D$. 
We use a
similar notation on tensor-product polynomial spaces on the unit cube. When $D:=B\otimes I$ is a unit prism having a triangular base $B$ with coordinates $(x,y)$ and a $z$-directional edge $I$, we denote by 
$\pol_{k|k}(D):= \pol_k(x,y)\otimes \pol_k(z)$ and $\widetilde{\pol}_{k|k}(D):=\widetilde\pol_k(x,y)\otimes \widetilde\pol_k(z)$ the space of 
tensor-product polynomials of degree no
greater than $k$, and the space of homogeneous tensor-product polynomials of degree $k$, respectively.
% on the prism $D$. 
Vector-valued spaces are denoted with a superscript $d$ (the space dimension); for example,
$\pol_k(K)^d$ is the space of vectors whose entries lie in $\pol_k(K)$.

 We denote by  $\|\cdot\|_{W^{m,p}(D)}$ the standard $W^{m,p}$-Sobolev norm on the domain $D\subset \bR^d$. For the Hilbert space 
$H^m(D):=W^{m,2}(D)$,  we simply write  $\|\cdot\|_{m,D}$ instead of
$\|\cdot\|_{H^{m}(D)}$, and $\|\cdot\|_{D}$ instead of $\|\cdot\|_{0,D}$.
Similarly, when $p=\infty$, we write $\|\cdot\|_{m,\infty, D}$ instead of
$\|\cdot\|_{W^{m,\infty}(D)}$, and $\|\cdot\|_{\infty, D}$ instead of $\|\cdot\|_{0,\infty, D}$.
For a given a second-order tensor $\mathrm{c}$, we denote by 
$\|\cdot\|_{\mathrm{c},D}$ the $\mathrm{c}$-weighted $L^2$-norm on the domain $D$.

\ber{Finally, we denote by $\lmax(K)$ the $L^\infty(K)$-norm of the maximum eigenvalue of the tensor $\mathrm{c}$}.

\subsection{\!Examples of spaces $\bld{V}(K)\times W(K)$ admitting $M(\dK)$-decompositions}
An \mm-decomposition relates the trace of the normal component of the space of approximate fluxes $\VV(K)$
%\subset \left\{\vv\in L^2(K)^d:\divs \vv\in L^2(K),\vv\cdot\n|_{\dK}\in L^2(\dK)\right\},$
% $\VV\subset 
% \{\vv\in L^2(K)^d:\;\divs\vv\in L^2(K),\vv\cdot\n|_{\dK}\in L^2(\dK)\}$ 
and 
the trace of the space of approximate scalars 
$\WW(K)$%\subset H^1(K)$
% $W\subset H^1(K)$ 
with the space of approximate traces  
$
M(\partial K).%:=\{\wwhat\in L^2(\partial K): \;\wwhat|_F\in M(F)\forall\;F\in \mathcal{F}(K)\}.
$  
To define it, 
we need to consider the combined trace operator
\begin{alignat*}{6}
\tr : & \VV(K)\times W(K) &\quad\longrightarrow\quad 
& L^2(\partial K)\\
& (\vv,w) &\quad\longmapsto\quad & 
(\vv\cdot\n +w)|_{\partial K}
\end{alignat*}

\begin{definition} [The \mm-decomposition] We say that $\VV(K)\times \WW(K)$ admits an \mm-decomposition when
\label{definition:m}
\begin{itemize}
\item[{\rm(a)}] $\tr (\VV(K)\times \WW(K))\subset M(\partial K)$,
\end{itemize}
and there exists a subspace $\Vtilde(K)\times\Wtilde(K)$ of $\VV(K)\times W(K)$ satisfying
\begin{itemize}
\item[{\rm(b)}] $\nabla \WW(K)\times\nabla\cdot\VV(K)\subset \Vtilde(K)\times\Wtilde(K),$
\item[{\rm(c)}] $\tr: \Vperp(K)\times \Wperp(K)\rightarrow M(\partial K)$ is an isomorphism.
\end{itemize}
Here $\Vperp(K)$ and $\Wperp(K)$ are the {$L^2(K)$-}orthogonal complements of 
$\Vtilde(K)$ in $\VV(K)$, and of $\Wtilde(K)$ in $\WW(K)$, respectively.
\end{definition}

Local spaces $\VV(K)\times \WW(K)$ admitting $M(\partial K)$-decompositions have been explicitly constructed in two-dimensions 
for general polygonal elements $K$ (see some examples in Table \ref{table:examples2D}) in 
\cite{CockburnFuM2D} and in three-dimensions for four types of polyhedral elements $K$, namely, tetrahedra, pyramids, prisms, and hexahedra
in \cite{CockburnFuM3D}. \ber{As pointed out in the Introduction, the} main interest of these spaces is that they generate superconvergent HDG and mixed methods, see \cite{CockburnFuSayas16}.

\begin{table}[ht]
\caption{Spaces $\VV(K)\times W(K)$ admitting an $M(\partial K)$-decomposition. {\rm \cite{CockburnFuSayas16}}}
 \centering
\begin{tabular}{l c c }
\hline
\noalign{\smallskip} 
         \multicolumn{1}{c}{$\boldsymbol{V}(K)$}         &     $W(K)$        &     method \\
\noalign{\smallskip}
\hline\hline
\noalign{\smallskip} 
\multicolumn{3}{c}{$M(\dK)=\pol_{k}(\partial K)$, $K$ is a square. % and $\VVV\times\WWW=\bqol_k\times\qol_k$.
 }\\
\hline
\noalign{\smallskip}
  $\bqol_k\oplus\bld{\mathrm{curl}}\;\mathrm{span}\{
x^{k+1} y , x\,y^{k+1}\}\oplus\;\mathrm{span}\{
\bld x\, x^{k} y^k\}$ & $\qol_{k}$&  ${\mathbf{TNT}_{[k]}}$ \cite{CockburnQiuShi12}\\
  $\bqol_{k}\oplus\bld{\mathrm{curl}}\;\mathrm{span}\{
x^{k+1} y , x\,y^{k+1}\}$ &  $\qol_{k}$ & ${\mathbf{HDG}^Q_{[k]}}$\cite{CockburnQiuShi12} \\
  $\bqol_k\oplus\bld{\mathrm{curl}}\;\mathrm{span}\{
x^{k+1} y , x\,y^{k+1}\}$& $\qol_{k}\setminus\{x^k\,y^k\}$ &  ${\mathbf{BDM}_{[k]}}$\\
\noalign{\smallskip}
\hline 
\noalign{\smallskip} 
\multicolumn{3}{c}{$M=\pol_{k}(\partial K)$, $K$ is a triangle. %and  $\VVV\times\WWW=\bpol_k\times\pol_k$.
}\\
\hline
\noalign{\smallskip}
  $\bpol_k\oplus\boldsymbol{x}\,\widetilde{\pol}_k$ & $\pol_{k}$&  ${\mathbf{RT}_{k}}$ \cite{RaviartThomas77}\\
 $\bpol_{k}$ &  $\pol_{k}$ & $\boldsymbol{\mathrm{HDG}}_k$\cite{CockburnQiuShi12}\\
 $\bpol_k$& $\pol_{k-1}$ 
&  ${\mathbf{BDM}_{k}}$ \cite{BrezziDouglasMarini85}\\
\noalign{\smallskip}
\hline 
\noalign{\smallskip} 
\multicolumn{3}{c}{$M=\pol_{k}(\partial K)$, $K$ is a square.% and $\VVV\times\WWW=\bpol_k\times\pol_k$.
 }\\
\hline
\noalign{\smallskip}
  $\bpol_k\oplus\bld{\mathrm{curl}}\;\mathrm{span}\{
x^{k+1} y , x\,y^{k+1}\}\oplus\;\boldsymbol{x}\,\widetilde{\pol}_k$ & $\pol_{k}$&  \cite{CockburnFuM2D}\\
 $\bpol_{k}\oplus\bld{\mathrm{curl}}\;\mathrm{span}\{
x^{k +1} y , x\,y^{k+1}\}$ &  $\pol_{k}$ &\cite{CockburnFuM2D}\\
 $\bpol_k\oplus\bld{\mathrm{curl}}\;\mathrm{span}\{
x^{k+1} y , x\,y^{k+1}\}$& $\pol_{k-1}$ 
&  ${\mathbf{BDM}_{[k]}}$ \cite{BrezziDouglasMarini85}\\
\noalign{\smallskip}
\hline 
\noalign{\smallskip} 
\multicolumn{3}{c}{$M=\pol_{k}(\partial K)$, $K$ is a quadrilateral.% and  $\VVV\times\WWW=\bpol_k\times\pol_k$.
}\\
\hline
\noalign{\smallskip}
  $\bpol_k\oplus_{i=1}^{ne}\bld{\mathrm{curl}} \,\mathrm{span}\{
\xi_4\,\lambda_3^k , \xi_4\,\lambda_4^k\}
  \oplus 
\boldsymbol{x}\,\widetilde{\pol}_k
$ & $\pol_{k}$ & \cite{CockburnFuM2D}\\
 $\bpol_{k}\oplus_{i=1}^{ne}\bld{\mathrm{curl}} \,\mathrm{span}\{
\xi_4\,\lambda_3^k , \xi_4\,\lambda_4^k\}\
 $ &  $\pol_{k}$& \cite{CockburnFuM2D}\\
 $\bpol_k\oplus_{i=1}^{ne}\bld{\mathrm{curl}} \,\mathrm{span}\{
\xi_4\,\lambda_3^k , \xi_4\,\lambda_4^k\}\
 $& $\pol_{k-1}$ & \cite{CockburnFuM2D}\\
\noalign{\smallskip}
\hline 
\end{tabular}
\label{table:examples2D}
\end{table}
Let us explain the notation used in the above table. By $\bld{\mathrm{curl}}\,p$  we mean the vector $(-p_y,p_x)$. 
\ber{By $\{\vt_i\}_{i=1}^{4}$ (and $\vt_5:=\vt_1$), we mean the four vertices of a quadrilateral; the vertices are ordered in a counter-clockwise manner. We} 
denote by $\eg_i$ the edge connecting the vertices $\vt_i$ and $\vt_{i+1}$.  Then, we set
\begin{align*}
 \xi_i := &\; \eta_{i-1}\frac{\lambda_{i-2}}{\lambda_{i-2}(\vt_i)}+
 \eta_{i}\frac{\lambda_{i+1}}{\lambda_{i+1}(\vt_{i})}
 \quad\mbox{ and }\quad
 \eta_i :=\Pi_{\underset{ j\not=i}{j=1}}^{4} \frac{\lambda_j}{\lambda_j+\lambda_i},
\end{align*}
where $\lambda_i$ \ber{is} the  linear function that 
vanishes on \ber{the} edge $\eg_i$ and reaches the maximum value $1$ in the closure of $K$. For details, see \cite{CockburnFuSayas16,CockburnFuM2D}.

\subsection{The stabilization subspace $M_S(\partial K)$}
\label{subsec:dh1-stabilization}

We also need to introduce  the stabilization space $M_S(\partial K)$.
\ber{This} is a subspace of $M(\dK)$ satisfying the following two conditions inspired from \cite[Proposition 3.2]{CockburnFuSayas16}:
\begin{subequations}
\label{condition-ms}
\begin{align} 
\dim M_S(\dK) = \dim \widetilde{W}^\perp(K) = &\;\dim \WW(K)-\dim \divs \VV(K),\\
 \|P_{M_S}(\cdot )\|_{\dK}\text{  is a norm} & \text{ on the space } 
 \widetilde{W}^\perp(K). 
\end{align}
\end{subequations}
Here, $P_{M_S}$ denotes the $L^2(\dK)-$projection into the space $M_S(\dK)$.
\ber{Examples of $M_S(\partial K)$ for various element shapes are collected in the following proposition, 
whose proof is given in Section 3.}

% \begin{itemize}
%  \item $\dim M_S = \dim \Wbd = \dim \WW-\dim \divs \VV$. 
%  \item $\|P_{M_S}(\wwhat)\|_{\dK}$ is a norm on the space $\Wbd$. 
% \end{itemize}

% \begin{remark} 
% \em
% Note that the existence of the stabilization space $M_S(\dK)$ is guaranteed since we can always choose it to be $\Wbd(K)$ itself, 
% \textcolor{black}{where $\gamma$ is the trace operator on $\partial K$ of scalar functions}. 
% However, this is not a particular efficient choice for implementation
% since it requires solving 
% a linear system of dimension $\dim \Wbd(K)$ \ber{for  the} evaluation of $\alpha (\wwhat)$ for any function $\wwhat\in M(\dK)$. 
% Note also that, to simplify the implementation, it would be very convenient to have all 
% the basis functions for $M_S(\dK)$ having support on a single face, not on the whole 
% boundary $\dK$. This prompts the construction of spaces $M_S(\dK)$ 
% with support on a single face to minimize the evaluation of the stabilization function 
% $\alpha (\wwhat)$ for $\wwhat \in M(\dK)$.
% Such construction is realized in the next proposition for various spaces admitting 
% \mm-decompositions\ber{. Its} proof is given in Section 3.1 below.
% \end{remark}

\begin{proposition}
\label{prop:ms}
Let the space $\VV(K)\times \WW(K)$ admit an $M(\dK)$-decomposition.
Then, conditions \eqref{condition-ms} are satisfied 
\begin{itemize}
 \item [\em{(1)}] If $\divs \VV(K)=\WW(K)$ and $M_S(\dK)=\emptyset$.
 \item [\em{(2)}] If $\divs \VV(K)=\pol_{k-1}(K), \ber{\WW(K)=\pol_k(K)}$
 and
 \vskip-.5truecm  
\[
  M_S(\dK):=\{\wwhat\in L^2(\dK):\;\;\wwhat|_{F^*}\in \pol_k(F^*), \;\;\wwhat|_{\dK\backslash F^*} = 0 \}
 \] 
 \vskip-.2truecm  
Here $F^*$ is a fixed  face of the element $K$ such that 
$K$ \ber{lies} in one side of the hyperplane containing $F^*$.
 \item [\em{(3)}] If $K$ is a square or cube,  
 $\divs \VV(K)=\divs \qol_k(K)^d, \ber{\WW(K)=\qol_k(K)}$
 and
\vskip-.5truecm  
 \[
  M_S(\dK):=\{\wwhat\in L^2(\dK):\;\;\wwhat|_{F^*}\in \widetilde{\qol}_k(F^*), 
  \;\;\wwhat|_{\dK\backslash F^*} = 0 \}.
 \] 
\vskip-.2truecm  
 Here $F^*$ is any fixed face of the square or cubic element $K$.
  \item [\em{(4)}] If $K$ is a prism with tensor product structure,
 $\divs \VV(K)=\divs \pol_{k|k}(K)^d, \linebreak \ber{\WW(K)=\pol_{k|k}(K)}$,
 and
\vskip-.5truecm  
 \[
  M_S(\dK):=\{\wwhat\in M:\;\;\wwhat|_{F^*}\in \widetilde{\pol}_k(F^*), 
  \;\;\wwhat|_{\dK\backslash F^*} = 0 \}.
 \] 
\vskip-.2truecm  
 Here $F^*$ is a  \ber{triangular base} of the prism $K$.
\end{itemize}
\end{proposition}

\subsection{Discrete $H^{1}$- and Poincar\'e-Friedrichs inequalities}
Our main result is the following. 

\begin{theorem}[Local, discrete $H^1$- and Poincar\'e-Friedrichs inequalities]
 \label{thm:dh1}
 Let $K$ be any element of the mesh $\mathcal{T}_h$. 
 Consider the mapping $(u_h,\,\uhat)\in W(K)\times M(\partial K) \longmapsto \;\bld{q}_h\in \bld{V}(K)
$ given by  \eqref{gc_hdg}. Then, if $\bld{V}(K)\times W(K)$ 
admits an $M(\partial K)$-decomposition, and
 \[
 \Theta_K:=\left( \lmax(K) \,\|\bld{q}_h\|_{\mathrm{c},K}^2 + 
%   \bintK{\alpha(u_h-\uhat)}{u_h-\uhat}
  h_K^{-1}\|P_{M_S}(u_h-\uhat)\|_\dK^2\right),
 \]
where 
$M_{S}(\partial K)$ is any subspace of  $M(\partial K)$
 satisfying conditions \eqref{condition-ms}, we have the inequalities
 \begin{alignat*}{2}
%  \label{m-decomposition-dh1}
  |\,(u_h,\widehat{u}_h)\,|_{1,K}^2&\le C\,\Theta_K
  &&\quad\mbox{{\rm ($H^1$)},}
  \\
 % \label{m-decomposition-dl2}
  h_K^{-2}\, |\,(u_h,\widehat{u}_h)\,|_{\mbox{{\rm \tiny PF}},K}^2&\le C\,\Theta_K&&\quad\mbox{{\rm (Poincar\'e-Friedrichs)},}
 \end{alignat*}
where
the constant $C$ only depends on the finite element spaces $\VV(K)$, $\WW(K)$ and $M_S(\dK)$, 
and on the shape-regularity properties of the element  $K$.
\end{theorem}

\

A detailed proof of this result is given in the next section. Here,  let us briefly discuss it:

\

(1). First, note that it is not very difficult to obtain these inequalities if the projection operator
$P_{M_S}$ is replaced by the identity. Indeed, if we {\em only} assume that $\nabla W(K)\subset \bld{V}(K)$,  we can take $\bld{v}:=\nabla u_h$ in the equation defining $\bld{q}_h$, \eqref{gc_hdg_o}, to immediately obtain 
\[
\|\nabla u_h\|^2_K \le \|\mathrm{c}\,\bld{q}_h\|^2 _K + C\,h^{-1}_K\,\|u_h-\uhat\|^2_{\partial K}.
\]
The wanted inequality now easily follows.
\ber{However, such choice might degrade the accuracy of the HDG method, as is typical of DG methods, see, for example, \cite{CastilloCockburnPerugiaSchoetzau00}. To avoid this, we must
chose a {\it minimal} space $M_S$ such that the inequalities in Theorem 2.3 still hold.}
% In other words,  the smaller the stabilization space $M_S(\dK)$, the more difficult it is to obtain the estimates.

(2). The inequalities of the above result are nothing but {\em stabilized} versions of $\inf$-$\sup$ conditions for the bilinear form defining $\bld{q}_h$, see \eqref{gc_hdg}, since 
\[
\| \mathrm{c}\,\bld{q}_h\|_K \ge \sup_{\bld{v}\in\bld{V}(K)\setminus\{\bld{0}\}}\frac{(u_h,\nabla\cdot\bld{v})_K-\langle\uhat, \bld{v}\cdot\bld{n}\rangle_{\partial K}}{\| \bld{v}\|_K},
\]
see \cite[Section 6.3]{BoffiBrezziFortin13}. For this reason, the subspace $M_S(\dK)$ is called a {\em stabilization} subspace. 
% In fact, 
% the role of the subspace $M_S(\dK)$ is to render the above inequalities
% valid for functions $(u_h,\uhat)$ for which $\bld{q}_h=\bld{0}$ and such that the seminorms $| (u_h, \uhat)|_{1,K}$ and $ | (u_h, \uhat)|_{\mbox{{\rm \tiny PF}},K}$ are not zero. 

\

(3). Let us argue that the dimension of the stabilization space $M_S(\dK)$ is actually minimal. It is obvious {that the influence of $u_h$ on $\bld{q}_h$ is only through its
$L^2$-projection} into $\nabla\cdot\bld{V}(K)$. 
As a consequence, the part of $u_h$ lying on the 
$L^2(K)$-orthogonal complement of $\nabla\cdot\bld{V}(K)$ in $W(K)$ {\em cannot} be controled by the size of $\bld{q}_h$. Since the dimension of such space is 
$\dim \WW(K)-\dim \divs \VV(K) $ and this number, by the first of conditions \eqref{condition-ms}, is equal to $\dim M_S(K)$, we see that the dimension 
of $M_S(\dK)$ cannot be smaller for the inequalities under consideration to hold.

\
 
(4).
% As pointed out in the Introduction, 
\ber{The above $H^1$-inequality has been explicitly obtained in the literature for two cases \cite{EggerSchoberl10,ChungEngquist09}. 
The first \cite{EggerSchoberl10} is the 
case of the Raviart-Thomas elements on a simplex in which the spaces, using our notation, 
\begin{align*}
\VV(K)=&\pol_k(K)^d+ \bld{x}\,\pol_k(K), \quad\WW(K):=\pol_k(K),\\ 
M(\partial K):=&\{\mu\in L^2(\dK):\;\mu|_F\in\pol_k(F) \;\forall\; F\in\mathcal{F}(K)\},
M_S(\dK)=\emptyset, 
\end{align*}
see \cite[Proposition 3.2]{EggerSchoberl10};
% It holds for $M_S(\dK):=\{0\}$. 
the second \cite{ChungEngquist09} is the case for the 
staggered DG method in which the spaces (defined on a simplex) are given as follows:
\begin{align*}
\VV(K)=&\pol_k(K)^d, \quad\WW(K):=\pol_k(K),\\ 
M(\partial K):=&\{\mu\in L^2(\dK):\;\mu|_F\in\pol_k(F) \;\forall\; F\in\mathcal{F}(K)\},\\
M_S(\dK):=&\{\mu\in M(\partial K): \mu=0 \mbox{ on }\partial K\setminus F_K\},
\end{align*}
% , where the same spaces for $W(K)$ and $M(\dK)$ is used, but 
% the local Raviart-Thomas space is reduced to 
% the local Brezzi-Douglas-Marini space $\VV(K) = \pol_k(K)^d$, and the stabilization space enriched to be 
% % which we modify the space of fluxes and take the smaller space $\VV(K):= \pol_k(K)^d$. It holds for
% \[
% M_S(\dK):=\{\mu\in M(\partial K): \mu=0 \mbox{ on }\partial K\setminus F_K\},
% \] 
where $F_K$ is a single face of the simplex $K$; see \cite[Theorem 3.2]{ChungEngquist09}.
% , where they set
% $\uhat:=u_h$ on $F_K$.
}

\

(5). 
\ber{Given data $\widehat u_h$ and $f$, 
let $(\bld{q}_h,u_h)\in \bld{V}(K)\times W(K)$ be the solution to the local problem
\eqref{HDG equations a}--\eqref{HDG equations b}, 
with the space $\bld{V}(K)\times W(K)$ admitting an $M(\dK)$-decomposition.
The following inequalities were obtained in \cite[Theorem 4.3]{CockburnFuSayas16}
\begin{alignat*}{1}
\|\nabla u_h\|^2_K &\le C\,\left( \lmax(K) \,\|\bld{q}_h\|_{\mathrm{c},K}^2 +\|P_{\widetilde{W}^\perp} f\|^2_K\right),
\\
h_K^{-1}\|u_h-\uhat\|^2_{\dK} &\le C\,\left( \lmax(K) \,\|\bld{q}_h\|_{\mathrm{c},K}^2 +\|P_{\widetilde{W}^\perp} f\|^2_K\right).
\end{alignat*}
Our result replaces the quantity $\|P_{\widetilde{W}^\perp} f\|^2_K$ on the above right hand side with 
\[h_K^{-1}\|P_{M_S}(u_h-\uhat)\|_\dK^2.\]
It is this small change that significantly facilitates the analysis of HDG schemes for the incompressible 
Navier-Stokes equation considered in this paper. %By using the techniques used in \cite{CockburnFuSayas16} can not be applied therein.
% The appearance of the term $\|P_{\widetilde{W}^\perp} f\|^2$ in these inequalities is not a desirable 
% When $f$ \ber{is a very} smooth function, the term  
% $\|P_{\widetilde{W}^\perp} f\|_K=\|{P_W}f-P_{\widetilde{W}}f\|_K$
% is very small since, when $\bld{V}(K)\times W(K)$ admits an $M(\dK)$-decomposition, $\widetilde{W}(K)=\nabla\cdot\bld{V}(K)$, or simply zero, whenever 
% ${W}(K)=\nabla\cdot\bld{V}(K)$. 
}
% \begin{subequations}
% \label{HDG-diffusion}
% \begin{alignat}{2}
% \label{HDG-diffusion_1}
% (\mathrm{c}\,\bld{q}_h,{\bld v})_{K} 
% - (u_h,\nabla \cdot {\bld v})_{K} 
% + \langle \uhat,{\bld v} \cdot {\bld n}\rangle_{\dK} & = 0
% &&\quad\forall\; \bld{v}\in\bld{V}(K),\\
% \label{HDG-diffusion_2}
% (\bld{q}_h,\nabla w)_{K} 
% - \langle \qhat \cdot {\bld n},w\rangle_{\dK} &= ({f},w)_{K}
% &&\quad\forall\; w\in W(K), 
% \end{alignat}
% \end{subequations}
% where $
% \qhat{}\cdot\bld{n}:= \bld{q}_h\cdot\bld{n} + \alpha(u_h-{\uhat})$ on $\partial K$, 
% if we use a space $\bld{V}(K)\times W(K)$ admitting an $M(\dK)$-decomposition, the second and third inequalities of 

\

% \ber{(6.) The proof of the discrete $H^1$-inequality presented in the next section} can be considered as variation of the proof of the above inequalities,
% as we can consider that  $f$ is  {\em defined} by the second equation (\ref{HDG-diffusion_2}) of the local problem. However, 
% the proof we present is new and self-contained. Moreover, it does {\em not} use the second equation and so, it could be 
% used for analyzing general convection-diffusion (-reaction) problems as we 
% comment in the last {s}ection. \ber{To prove the discrete Poincar\'e-Friedrichs inequality, we rely on the} characterization of \mm-decompositions presented in \cite[Theorem 2.4]{CockburnFuSayas16} which states that we have the identity
% \[
% \{\mu\in M(\dK):\;\langle \mu, 1\rangle_{\dK} \textcolor{black}{ = 0 } \}=\{\vv\cdot\bld{n}|_{\dK}:\;
%  \vv\in \bld{V}(K), \;\nabla\cdot\vv=0\}.
% \]

\ber{(6). The dependence of the constant $C$ in the estimates on the local spaces 
$\boldsymbol{V}(K), W(K), \text{ and } M_S(\partial K),$ 
and on the shape regularity of the element $K$ remains to be studied. It is reasonable to believe that $C$ can be uniformly bounded by  
a function of the the maximum degree of the polynomial functions belonging to the local spaces and by a suitable measure of the element shape-regularity.
}

\subsection{Choosing the stabilization function $\alpha$ to get $H^1$-stability} 
We end this Section by illustrating the fact that the stabilization subspace $M_S(\dK)$ 
can be actually used, when defining HDG methods,  to obtain what we could call the {\em minimal} 
stabilization function $\alpha$ needed to achieve a new $H^1$-stability result. Let us do that in the framework of  HDG approximations for steady-state diffusion problems. 

So, if $(\bld{q}_h,u_h)\in \bld{V}(K)\times W(K)$ 
is the solution of the local problem \eqref{HDG equations a}--\eqref{HDG equations b},
  we have the discrete energy identity
\[
\mathsf{E}_K(\bld{q}_h;u_h,\uhat)
=(f, u_h)_K -\langle \qhat\cdot\bld{n}, \uhat\rangle_{\dK},
\]
where
\vskip-1truecm
\[%\begin{equation}
%\label{gc_hdg_enery_o}
\mathsf{E}_K(\bld{q}_h;u_h,\uhat):=
(\mathrm{c}\,\bld{q}_{h}, \bld{q}_h)_K + \langle \alpha(u_h - \uhat), u_h - \uhat\rangle_{\dK},
%\end{equation}
\]
is the {\em energy} associated to the element $K$. We immediately \ber{see} that
\[
\| \mathrm{c}\,\bld{q}_h\|^2 _K + h^{-1}_K\,\|P_{M_S}(u_h-\uhat)\|^2_{\partial K}\le C\, \mathsf{E}_K(\bld{q}_h;u_h,\uhat),
\]
if we pick the stabilization function $\alpha$ as
\begin{equation}
\label{alpha}
\alpha(\widehat{\omega}):={h_K^{-1}}\,P_{M_S}(\widehat{\omega})\;\;\forall\;\widehat{\omega}\in L^2(\dK),
\end{equation}
% with some \ber{positive} parameter $\eta_K$ of order one, 
case in which we  say that this stabilization function $\alpha$ is {\em minimal}. 
Thus, by establishing this link between the HDG stabilization function $\alpha$ and 
the stabilization subspace $M_S(\dK)$, an estimate of the energy immediate implies 
an estimate on the discrete seminorms under consideration, that is,
\[
{\max\{ h_K^{-2}|(u_h,\uhat)|^2_{\mbox{{\rm \tiny PF}},K},  |(u_h,\uhat)|^2_{1,K}\}}\le C \,\mathsf{E}_K(\bld{q}_h;u_h,\uhat).
\]
%\textcolor{black}{
%where $|(\cdot,\cdot)|_{1,K}$ is defined by \eqref{localseminorms},
%and
%\begin{alignat*}{1}
%|(v,\widehat{v})|^2_{0,K}:= \|v\|^2_K+ h_K(\|\widehat{v}\|^2_{\dK}+\|v-\widehat{v}\|^2_{\dK}).
%\end{alignat*}
%}
% Note that the last inequality does hold even if we take $\eta_K$ of order $h^{-s}_K$
% for {\em any} positive number $s$. This is exactly what happens for the so-called SFH method proposed in \cite{CockburnDongGuzman08}. We could even consider the limit case in which $\eta_K$ goes to infinity, which is what the so-called SDG method \cite{ChungEngquist09} does. The fact that the SDG method can be seen as a limiting case of as SFH method was shown in \cite{ChungCockburnFu14}. 
\ber{Now consider the full HDG scheme \eqref{HDG equations} for diffusion, 
we easily obtain discrete $H^1$-stability result of the approximation with respect to the data $f$ by summing 
the above inequality over all elements:
% \ber{Indeed, s}ince the problem is linear, to obtain this result, we assume that the Dirichlet boundary condition is identically zero and set $\uhat=0$ on $\partial \Omega$. Then
% after adding over all the elements, we can easily get that
\[
\vertiii{(u_h,\uhat)}_{1,\Oh}^2
=\sum_{K\in\Oh} |(u_h,\uhat)|^2_{1,K}
\le C  \sum_{K\in\Oh} \mathsf{E}_K(\bld{q}_h;u_h,\uhat)
= C\, (f,u_h)_\Oh.
\]
% and, since $\|u_h\|_{\Omega}\le C'\,\vertiii{(u_h,\uhat)}_{1,\Oh}$ whenever $\uhat|_{\partial\Omega}=0$, see, for example, \cite[Theorem 2.1]{DiPietroDroniouErn10} ,
% we obtain the $H^1$-stability result
% \[
% \vertiii{(u_h,\uhat)}_{1,\Oh}\le  CC'\, \| f \|_{\Oh}.
% \]
% Using this stability result, the 
% analysis of the resulting methods becomes much simpler than, for example, the one proposed in \cite{ChenCockburn12},  since
% the auxiliary projection for the analysis does not need to involve the convective derivative. We 
% use this approach when dealing with the HDG and mixed methods for the Navier-Stokes equations.
This stability result can be \ber{similarly} obtained for the HDG method for 
the convection-difussion equation in which convection is treated with the standard {\it upwinding} technique.
We use this approach in Section 4 to deal with the HDG and mixed methods for the Navier-Stokes equations.
%  \ber{Then, 
% the last inequality holds for $\alpha:=\alpha_d+\alpha_{upw}$, where the stabilization due to the difussion $\alpha_d$ would be given by \eqref{alpha}, and,  the stabilization for the convection, $\alpha_{upw}$ would be, for
% example, the usual {stabilization by upwinding}}
}

% Note also that, when convection is added, the form of the energy is actually the {\em same}. \ber{Then, 
% the last inequality holds for $\alpha:=\alpha_d+\alpha_{upw}$, where the stabilization due to the difussion $\alpha_d$ would be given by \eqref{alpha}, and,  the stabilization for the convection, $\alpha_{upw}$ would be, for
% example, the usual {stabilization by upwinding}}. In this case, we obtain stability results of the approximation
% with respect to the data $f$.
 
\section{Proofs of the results of Section 2}
\label{sec:proof-dh1}
In this Section, we give a proof of the  properties of the stabilization spaces $M_S(\dK)$,
and then a proof of the discrete $H^1$- and the discrete Poincar\'e-Friedrichs inequalities.

% \subsection{Proof of the properties of the stabilization spaces $M_S(\dK)$}
\subsection{\ber{Proof of Proposition \ref{prop:ms}}}
\ber{Let us first prove Proposition \ref {prop:ms} on the properties of the stabilization spaces $M_S(\dK)$.}
% Here, we give a proof of .}
We just prove the second case since the proofs for the other three are similar and simpler.

 For this case, we have $\divs \VV=\pol_{k-1}(K)$, $W=\pol_k(K)$ and
 \[
  M_S=\{\wwhat\in L^2(\dK):\;
  \wwhat|_{F^*} \in \pol_k(F^*), \wwhat|_{\dK\backslash F^*}=0\},
 \]
where $F^*$ is a face of the element $K$ such that $K$ \ber{lies} on one side of the hyperplane 
containing $F^*$.  Hence, we have 
\begin{align*}
 \dim M_S &\;= \dim \pol_k(F^*)
 = \dim \pol_k(K) - \dim \pol_{k-1}(K)\\
&\; = \dim W-\dim \divs \VV
= \dim W-\dim \Wtilde
 = \dim \gamma(\Wperp).
\end{align*}
This proves the first condition for $M_S$.

To prove the second condition, we only need to show that for any function 
$\wwhat \in \gamma(\Wperp)$, $P_{M_S}(\wwhat)=0$ implies $\wwhat=0$. 
Now, let $\wwhat$ be a function in $\gamma(\Wperp)$ such that $P_{M_S}(\wwhat)=0$.
By the definition of $\gamma(\Wperp)$, there exists a function $w\in \Wperp$ such that 
$\gamma(w) = \wwhat$. Hence, $P_{M_S}(\gamma(w))=0$. By the definition of $M_S$ and $W$, 
we have $w = \lambda \widetilde{w}$ where $\lambda\in \pol_1(K)$ is the linear function vanishing on $F^*$ and 
$\widetilde{w}\in\pol_{k-1}(K)= \Wtilde$. By $L^2$-orthogonality of the spaces $\Wtilde$ and $\Wperp$, we have 
\[
 (w,\widetilde{w})_K =  (\lambda \widetilde{w},\widetilde{w})_K=0,
\]
which immediately implies $w = 0$ by the assumption on the face $F^*$. This completes the proof 
of Proposition \ref{prop:ms}.

% \subsection{Proof of the discrete inequalities}
\subsection{\ber{Proof of Theorem \ref{thm:dh1}}}
\label{subsec:thm:dh1}
\ber{Here, we prove the inequalities of Theorem \ref{thm:dh1}.
Although it is enough to prove only one since the seminorms $|(\cdot,\cdot)|_{1,K}$ and $|(\cdot,\cdot)|_{\mbox{{\rm \tiny PF}},K}$ are equivalent,} we provide a different proof for each of them, as they put in evidence different \ber{properties} of the
\mm-decompositions.

\subsubsection{Proof of the first inequality} To prove the first inequality, it is convenient to first carry out a simple integration-by-parts in the equation defining $\bld{q}_h$, \eqref{gc_hdg_o}:
 \begin{align*}
\ber{(\mathrm{c}\,\bld{q}_h, \vv)_K =- (\grads u_h, \vv)_K + \bintK{u_h-\uhat}{\vv\cdot\n}}\quad\forall\;\vv\in \VV(K).  
 \end{align*}
By Property (b) of an \mm-decomposition, we can now set $\vv:= \grads u_h$
to get
\begin{align*}
 \|\grads u_h\|_{K}^2 = &\;
 -(\mathrm{c}\,\bld{q}_h, \vv)_K + \bintK{u_h-\uhat}{\vv\cdot\n},
\end{align*}
and conclude that
\begin{alignat*}{1}
\|\grads u_h\|_{K}&
 \le (\lmax)^{1/2} \|\bld{q}_h\|_{\mathrm{c}, K}+ 
 C_{\grads \WW}\,{\color{blue}h_K^{-1/2}}\,\|u_h-\uhat\|_{\dK},
\\
  C_{\grads\WW} := &\;
\sup_{\vv\in \grads\WW(K)\backslash\{0\}} \frac{h_K^{1/2}\|\vv\cdot\n\|_\dK}{\|\vv\|_K}.
\end{alignat*}

Let us now estimate the jump $u_h-\uhat\in M(\dK)$.
By \ber{Property} (c) of an \mm-decomposition, we can write
that $u_h-\uhat =  P_{\Wbd}(u_h-\uhat) + P_{\Vbd}(u_h-\uhat)$. Now, by the second of conditions \eqref{condition-ms}, there is a constant $C_{M_S}$ such that
\begin{alignat*}{2}
 \|P_{\gamma\Wperp}(u_h-\uhat)\|_{\dK}
 \le &\;C_{M_S} \|P_{M_S}\,\big(P_{\gamma\Wperp}(u_h-\uhat)\big)\|_{\dK}\\
 \le &\;C_{M_S}\big( \|P_{M_S}\,(u_h-\uhat)\|_{\dK}
 +\|P_{M_S}\,\big(P_{\gamma\Vperp}(u_h-\uhat)\big)\|_{\dK}
 \big)\\
 \le &\;C_{M_S}\big( \|P_{M_S}\,(u_h-\uhat)\|_{\dK}
 +\|P_{\gamma\Vperp}(u_h-\uhat)\|_{\dK}
 \big).
\end{alignat*}
It remains to estimate $\|P_{\Vbd}(u_h-\uhat)\|_{\dK}$.
Taking $\vv\in \Vperp(K)$ such that $\vv\cdot\n|_{\dK} = P_{\gamma\Vperp}(u_h-\uhat)$
in the definition of $\bld{q}_h$, and using the fact that $\grads u_h\in \Vtilde(K)$ is $L^2$-orthogonal to 
$\vv\in \Vperp(K)$, we get 
\begin{alignat*}{2}
\|P_{\gamma\Vperp}(u_h-\uhat)\|_{\dK}^2 = &\;
(\mathrm{c}\,\bld{q}_h, \vv)_K, 
%\\ 
%\le &\; (\lmax)^{1/2}\|\bld{r}\|_{\mathrm{c}, K}\|\vv\|_K\\
%\le &\; C_{\Vperp}(\lmax)^{1/2}\|\bld{r}\|_{\mathrm{c}, K}\|P_{\gamma\Vperp}(u_h-\uhat)\|_{\dK},
\end{alignat*}
and conclude that
\begin{align*}
& \|P_{\gamma\Vperp}(u_h-\uhat)\|_{\dK}\le
 C_{\Vperp}(\lmax)^{1/2}\,{\color{blue}h_K^{1/2}}\,\|\bld{q}_h\|_{\mathrm{c}, K},
\\
& C_{\Vperp} :=
\sup_{\vv\in \Vperp(K)\backslash\{0\}} \frac{\|\vv\|_K}{h_K^{1/2}\|\vv\cdot\n\|_\dK}.
\end{align*}
The first inequality now easily follows.

\subsubsection{Proof of the second inequality} To prove this inequality, 
it is convenient \ber{to rewrite} the equation defining 
$\bld{q}_h$, \eqref{gc_hdg_o},  as follows:
 \begin{align*}
(\mathrm{c}\,\bld{q}_h, \vv)_K - (u_h-\overline{\uhat}^{\,\dK}, \nabla\cdot \vv)_K + \bintK{\uhat-\overline{\uhat}^{\,\dK}}{\vv\cdot\n} = 0\quad\forall\;\vv\in \VV(K). 
 \end{align*}
By \cite[Theorem 2.4]{CockburnFuSayas16}, since $\bld{V}(K)\times W(K)$ admits an $M(\dK)$ decomposition, we have the identity
 \begin{alignat}{1}
 \label{Mdec}
 \{\mu\in M(\dK):\;\langle \mu, 1\rangle_{\dK} \textcolor{black}{  = 0 } \}
 =\{\vv\cdot\bld{n}|_{\dK}:\;
 \vv\in \bld{V}(K), \;\nabla\cdot\vv=0\}.
 \end{alignat}
 This means that there is a function $\vv\in \bld{V}(K)$ such that $\vv\cdot\bld{n}|_{\dK}= \uhat-\overline{\uhat}^{\,\dK}$ and $\nabla\cdot\vv=0$. Using this function as test function, we get
\[
\|\uhat-\overline{\uhat}^{\,\dK}\|^2_{\dK}=-(\mathrm{c}\,\bld{q}_h, \vv)_K,
\]
and so,
\begin{alignat*}{1}
&\|\uhat-\overline{\uhat}^{\,\dK}\|_{\dK}\le (\lmax(K))^{1/2}\|\bld{q}_h\|_{c,K} C_{\bld{V}\cdot \bld{n}}\, h^{1/2}_K,
\\
&
C_{\bld{V}\cdot \bld{n}}:=\sup_{\footnotesize\begin{matrix}\mu\in M(\dK)\\
                                                     \langle\mu,1\rangle_{\dK}=0
                               \end{matrix}}
                     \inf_{\footnotesize\begin{matrix}\vv\in \bld{V}(K)\setminus\{0\}\\
                                                    \nabla\cdot\vv=0\\
                                                    \vv\cdot\bld{n}=\mu
                             \end{matrix}}
                             \frac{\|\vv\|_K}{h_K^{1/2}\|\vv\cdot\n\|_\dK}.
\end{alignat*}

\textcolor{black}{
It remains to estimate $\|u_h-\overline{\uhat}^{\,\dK}\|_K$. We define \ber{a} test function 
$\vv\in \bld{V}(K)$ such that $\nabla\cdot\vv=P_{\nabla\cdot\bld{V}} (u_h-\overline{\uhat}^{\,\dK})${, which we can assume to be different from zero}. 
Obviously, we get 
\begin{align*}
\| P_{\nabla\cdot \bld{V}}(u_h - \overline{\uhat}^{\,\dK})\|^2_{K} = & (\mathrm{c}\,\bld{q}_h, \vv)_K 
+ \bintK{\uhat-\overline{\uhat}^{\,\dK}}{\vv\cdot\n} , 
\end{align*}
and so, 
{\begin{align*}
&\| P_{\nabla\cdot \bld{V}}(u_h - \overline{\uhat}^{\,\dK})\|_{K} 
\leq  \big(\Vert \mathrm{c} \bld{q}_{h} \Vert_{K} 
 + h^{-1/2}_K\,\|\uhat-\overline{\uhat}^{\,\dK}\|_{\dK}\big)\,C_{\nabla\cdot\bld{V}}\,h_{K},
 \\
&C_{\nabla\cdot\bld{V}}:=\sup_{g\in \nabla\cdot\bld{V}(K)\setminus\{0\}}
                     \inf_{\footnotesize\begin{matrix}\vv\in \bld{V}(K)\\
                                                    \nabla\cdot\vv=g
                                                  \end{matrix}}
                             \frac{(\|\vv\|_K+h^{1/2}_K\|\vv\cdot\bld{n}\|_{\dK})}{h_K \|\nabla\cdot\vv\|_K}.
\end{align*}}
%\begin{align*}
%\| P_{\nabla\cdot \bld{V}}(u_h - \overline{\uhat}^{\,\dK})\|_{K} 
%\leq & \Vert \mathrm{c} \bld{q}_{h} \Vert_{K} \dfrac{\Vert \vv \Vert_{K}}{\Vert \nabla\cdot \vv \Vert_{K}}
% + \|\uhat-\overline{\uhat}^{\,\dK}\|_{\dK}
% \dfrac{\Vert \vv\cdot\n\Vert_{\partial K}}{\Vert \nabla\cdot \vv \Vert_{K}} \\
%\leq & C \big( \Vert \mathrm{c} \bld{q}_{h} \Vert_{K} + h_{K}^{-1/2}\|\uhat-\overline{\uhat}^{\,\dK}\|_{\dK}\big) 
%\cdot \dfrac{\Vert \vv \Vert_{K}}{\Vert \nabla\cdot \vv \Vert_{K}}.
%\end{align*}
%{Thus,} 
%\begin{alignat*}{1}
%&\| P_{\nabla\cdot \bld{V}}(u_h - \overline{\uhat}^{\,\dK})\|_{K}\;
%\le\; C \big((\lmax(K))^{1/2}\|\bld{q}_h\|_{c,K} 
%+  h_{K}^{-1/2}\|\uhat-\overline{\uhat}^{\,\dK}\|_{\dK}\big)\, h_K\,C_{\nabla\cdot\bld{V}},
%\\
%&
%C_{\nabla\cdot\bld{V}}:=\sup_{g\in \nabla\cdot\bld{V}(K)\setminus\{0\}}
%                     \inf_{\footnotesize\begin{matrix}\vv\in \bld{V}(K)\\
%                                                    \nabla\cdot\vv=g
%                                                  \end{matrix}}
%                             \frac{\|\vv\|_K}{h_K \|\nabla\cdot\vv\|_K}.
%\end{alignat*}
}
%It remains to estimate $\|u_h-\overline{\uhat}^{\,\dK}\|_K$. We construct the next test function $\vv$ as follows. 
%First, it is trivial to see that there is a function
%$\vv_1\in \bld{V}(K)$ such that 
%$\nabla\cdot\vv_1=P_{\nabla\cdot\bld{V}} (u_h-\overline{\uhat}^{\,\dK})$. By the identity \eqref{Mdec}, there is 
%a function $\vv_2\in \bld{V}(K)$ such that $\vv_2\cdot\bld{n}= \vv_1\cdot\bld{n}- \overline{\vv_1\cdot\bld{n}}^{\,\dK}$
%and $\nabla\cdot\vv_2=0$. This implies that $\vv:=\vv_1-\vv_2$ is a function in 
%$\bld{V}(K)$ such that $\nabla\cdot\vv=P_{\nabla\cdot\bld{V}} (u_h-\overline{\uhat}^{\,\dK})$ and constant $\vv\cdot\bld{n}$ 
%on $\dK$.
%Using this as a test function, we get
%\[
%\| P_{\nabla\cdot \bld{V}}(u_h - \overline{\uhat}^{\,\dK})\|^2_{K}= (\mathrm{c}\,\bld{q}_h, \vv)_K, 
%\]
%and so,
%\begin{alignat*}{1}
%&\| P_{\nabla\cdot \bld{V}}(u_h - \overline{\uhat}^{\,\dK})\|_{K}
%\le (\lmax(K))^{1/2}\|\bld{q}_h\|_{c,K}\, h_K\,C_{\nabla\cdot\bld{V}},
%\\
%&
%C_{\nabla\cdot\bld{V}}:=\sup_{g\in \nabla\cdot\bld{V}(K)\setminus\{0\}}
%                     \inf_{\footnotesize\begin{matrix}\vv\in \bld{V}(K)\\
%                                                    \nabla\cdot\vv=g
%                                                  \end{matrix}}
%                             \frac{\|\vv\|_K}{h_K \|\nabla\cdot\vv\|_K}.
%\end{alignat*}
{Finally, let us} estimate $(\mathrm{Id}-P_{\nabla\cdot \bld{V}})(u_h - \overline{\uhat}^{\,\dK})$. Since this function coincides with $P_{\widetilde{W}^\perp} (u_h - \overline{\uhat}^{\,\dK})$ because $\widetilde{W}(K)=\nabla\cdot\bld{V}(K)$, we get
\begin{alignat*}{2}
\| P_{\widetilde{W}^\perp} (u_h - \overline{\uhat}^{\,\dK})\|_K
\le &\; C_K\, h^{1/2}_K\,\|P_{\Wbd} (u_h - \overline{\uhat}^{\,\dK})\|_{\dK}
\\
\le &\; C_M\,C_K\, h^{1/2}_K\,\|P_{M_S} (u_h - \overline{\uhat}^{\,\dK})\|_{\dK}
\\
\le &\; C_M\,C_K\, h^{1/2}_K\,(
\|P_{M_S} (u_h - \uhat)\|_{\dK}
+\|P_{M_S} (\uhat - \overline{\uhat}^{\,\dK})\|_{\dK}),
\\
\le &\; C_M\,C_K\, h^{1/2}_K\,(
\|P_{M_S} (u_h - \uhat)\|_{\dK}
+\|\uhat - \overline{\uhat}^{\,\dK}\|_{\dK}),
\end{alignat*}
and the estimate follows.
This completes the proof of Theorem \ref{thm:dh1}.

\section{Application: HDG methods for the Navier-Stokes equations}
\label{sec:ns}
In this Section, we introduce and analyze
new HDG and mixed 
methods for the steady-state incompressible Navier-Stokes equation with 
velocity gradient-velocity-pressure formulation described by equations \eqref{ns-equation}.

We proceed as follows. After defining the methods, we show that their approximate solution exists, is unique and satisfies an
energy-boundedness property under \ber{a} {\em smallness} assumption \ber{on} the data. We then provide results on the convergence properties. 

Some of the errors involving the velocities are measured in
the norms and seminorms defined as follows.
For any $(\vv,\widehat{\vv}) \in \bld{V}_h\times \bld{M}_h$, we set
\begin{alignat*}{1}
\vertiii{(\vv,\widehat{\vv})}_{\ell,\Oh}^2 := \sum_{i=1}^d \sum_{K\in\Oh} |(\vv_i,\widehat{\vv}_i)|^2_{\ell,K}
\qquad\mbox{ for }\ell=0,1,{\mbox{\rm \tiny PF}},
\end{alignat*}
where $|(\cdot,\cdot)|_{1,K}$ and $|(\cdot,\cdot)|_{\mbox{\rm \tiny PF},K}$ are defined by \eqref{localseminorms},
and
\begin{alignat*}{1}
|(\vv_i,\widehat{\vv}_i)|^2_{0,K}:= \|\vv_i\|^2_K+ h_K(\|\widehat{\vv}_i\|^2_{\dK}+\|\vv_i-\widehat{\vv}_i\|^2_{\dK}).
\end{alignat*}

\subsection{Definition of the methods}
\label{subsec:ns-hdg}
\subsubsection{The general form of the methods}
The HDG and mixed methods for \eqref{ns-equation} seek an approximation to $(\ml, \bld{u}, p, \bld{u}|_{\Eh})$,
$(\ml_h, \boldsymbol{u}_h, p_h,  \muhat)$, in the space 
$\GG_h\times \Vh \times \Qh \times \bld {M}_h(0)$ given by
\begin{subequations}
\label{ns-hdg-space}
\begin{alignat}{3}
\GG_h:=&\;\{\mg\in{L}^2(\Oh)^{d\times d}:&&\;\mg|_K\in\GG(K),&&\; K\in\Oh\},
\\
\label{ns-hdg-space-v}
\Vh:=&\;\{\vv\in{L}^2(\Oh)^d:&&\;\vv|_K\in \VV(K),&&\; K\in\Oh\},
\\
\label{ns-hdg-space-q}
\Qh:=&\;\{\ber{q\in{L}^2(\Oh)}:&&\;q|_K\in Q(K),&&\; K\in\Oh, (q,1)_\Omega = 0\},
\\
\label{ns-hdg-space-m}
\bld M_h:=&\;\{\mwwhat\in{L}^2(\Eh)^d:&&\;\mwwhat|_F\in \bld M(F),&&\; F\in\Eh\},
\\
\bld M_h(0):=&\;\{\mwwhat \in \bld M_h:&&\;\mwwhat|_{\partial\Omega}=0\}.
\end{alignat}
\end{subequations}
where the local spaces $\GG(K), \VV(K), Q(K),$ and $\bld M(F)$ are suitably defined finite dimensional spaces, 
and determine it as the only solution of the following weak formulation:
\begin{subequations}
\label{ns-HDG-equations}
 \begin{alignat}{3}
 \label{ns-HDG-equations-1}
 \bint{\nu\,\ml_h}{\mg}+\bint{\bld u_h}{\nu\,\divv \mg}   - \bintEh{\muhat}{ \nu\,\mg\, \n} & = 0, \\
 \label{ns-HDG-equations-2}
 \bint{\nu\,\ml_h}{\gradv \vv} \ber{+ \bintEh{-\nu\,\ml_h\,\n
+\alpha_v(\bld u_h-\muhat) }{\vv - \mwhat}\;\;\hspace{0.6cm}} \nonumber&\\
- \bint{p_h}{\divs \vv} + \bintEh{p_h\,\n}{\vv - \mwhat}\;\;\hspace{3.3cm} \nonumber&\\
 -\bint{\bld u_h\otimes \bld{\beta}}{\gradv \vv} + \bintEh{
\ber{( \bld{\beta}
\cdot\n)}\,\muhat
+\alpha_c(\bld u_h-\muhat)
}{\vv - \mwhat} &= \bint{\bld f}{\vv},\\
 \label{ns-HDG-equations-3}
 -\bint{\bld u_h}{\grads q}+ \bintEh{\muhat\cdot\n}{ q} & \ber{= 0,}
\end{alignat}
\end{subequations}
for all $(\mg, \vv, q, \mwhat) \in  \GG_h\times \Vh \times \Qh\times \bld M_h(0)$, where
\[\alpha_v: L^2(\dK)^d\longrightarrow L^2(\dK)^d\;\;\; \text{ and} \;\;\;
\alpha_c: L^2(\dK)^d\longrightarrow L^2(\dK)^d
\] 
are the {\em local stabilization operators} related to the viscous and  convective parts, respectively.
To complete the definition of the method, we have to define the local spaces, the divergence-free post-processed 
velocity $\bld{\beta}$, and the stabilization operators. We do this next.

\subsubsection{{The local spaces}}
\label{subsec:ns-space}
% The construction of the local spaces we use here is essentially the one presented in 
% \cite{CockburnShiHDGStokes13}. Here, we follow the slight modification proposed in 
% \cite{CockburnFuQiu16}.
% The \ber{construction of the }local spaces for the HDG and mixed methods just introduced 
% is carried out in terms of local spaces originally found 
% for defining superconvergent HDG and mixed methods for the steady-state diffusion problem \ber{as follows}.
\ber{The finite element spaces are the ones used in \cite{CockburnFuQiu16} for Stokes flow.}
%Now, let us define local spaces for the HDG and mixed methods under consideration.  Let any trio 
\ber{Let the space}
$\VVD\times \WWD\times \MD$ be such that
$\VVD\times \WWD$ admits an $\MD$-decomposition, see Definition \ref{definition:m}. Moreover, we assume that
\begin{align}
 \label{require-w}
\WWD \text{ is a polynomial space \ber{such that}}
\sum_{i=1}^d \partial_{i}\WWD\subset \WWD.
\end{align}

Then, the local spaces $\GG(K)$, $ \VV(K)$, and $Q(K)$,
and the local trace space $\bld{M}(\dK)$ 
are defined as follows:
\begin{subequations}
\label{ns-local-spaces}
\begin{alignat}{2}
\label{ns-local-spaces-a}
 \GG_i(K)\times \VV_i(K)\times \bld{M}_i(K)
 := &\;\VVD\times \WWD\times \MD &&\quad i=1,\cdots, d,\\
\label{ns-local-spaces-b}
Q(K) := &\;\WWD.
\end{alignat}
\end{subequations}

\subsubsection{The post-processed velocity $\bld{\beta}$}
On the element $K$, the post-processed velocity $\bld{\beta}$ is taken in a \ber{finite dimentional} space ${\VV}^*(K)$ satisfying the conditions
\begin{subequations}
\label{div-space}
\begin{align}
\label{div-space-1}
&\VVD\subset {\VV}^*(K), \divs{\VV}^*(K) = \WWD,\\
\label{div-space-2}
&{\VV}^*(K)\times \WWD \text{ admits an $\MD$-decomposition}.
 \end{align}
\end{subequations}
This vector-valued  space can be easily constructed from $\VVD$, as shown in 
\cite[Proposition 5.3]{CockburnFuSayas16}. % \ber{For example, 
%for the spaces listed in Table \ref{table-example-m}, 
%${\VV}^*(K)$ can be simply taken to be $\VVD\oplus \bld{x}\,\widetilde{\pol}_k(K)$.}

On the element $K$,
the post-processed velocity $\bld{\beta}:= \mbeta(\bld{u}_h,\muhat)\in \VV^*_h $
is defined as the function in $\VV^*(K)$ such that 
\begin{subequations}
\label{post-process-defn}
 \begin{alignat}{2}  
 \label{post-process-defn-1}
 (\mbeta(\bld{u}_h,\muhat), \vv)_K = &\;
 (\bld{u}_h, \vv)_K&&\quad \forall\; \vv\in \widetilde{\VV^*}(K),\\
 \label{post-process-defn-2}
 \bintK{\mbeta(\bld{u}_h,\muhat)\cdot\n}{\widehat{v}}
 = &\;
  \bintK{\muhat\cdot\n}{\widehat{v}}
  &&\quad\forall\; \widehat{v}\in \MD.
 \end{alignat}
\end{subequations}
Here $\widetilde{\VV^*}(K):= \grads \WWD\oplus\{\vv\in\VV^*(K):\;
\divs \vv=0,\;\vv\cdot\n|_{\dK} = 0\}$.

\ber{ We gather the main properties of this mapping in the next result which we prove in Appendix \ref{sec:A}.}
%
%The following approximation property of the post-processing will be used in the analysis, whose proof is given 
%in Appendix B.
\begin{proposition}
 \label{lemma:div-proj}
 Let $(\bld v,\mwhat)\in \Vh\times \bld M_h$. Then, for any element $K\in\Oh$,  we have
 %\begin{subequations}
%\label{div-proj}
 \begin{alignat*}{2}
%\label{div-proj-1}
  \vertiii{(\mbeta{(\bld v,\mwhat),\ave{\mbeta{(\bld v,\mwhat)}})}}_{\ell,K}\le &\;
  C\,\vertiii{(\bld v,\mwhat)}_{\ell,K}&&\quad 
  \text{ for } \ell = 0,1,\\
%\label{div-proj-2}
 \|{\mbeta{(\bld v,\mwhat)}}\|_{\infty,K}\le &\;
  C\,\vertiii{(\bld v,\mwhat)}_{\infty,K},
 \end{alignat*}
 with a constant $C$ depending only on the space $\VV(K)\times \bld M(\dK)$ and the shape regularity of the element 
 $K$.
 Moreover, if $(\bld u_h,\muhat)\in \Vh\times \bld M_h(0)$ satisfies the weak incompressibility condition given by equation 
 \eqref{ns-HDG-equations-3}, 
 then %$\mbeta(\bld u_h,\muhat)\in \VV_{\!\!\beta}$, that is,
  \begin{alignat*}{2}
 %\label{div-proj-3}
%    \mbeta(\bld u_h,\muhat)\in \VV_{\!\!\beta}.
  \mbeta(\bld u_h,\muhat)\in H(\mathrm{div}, \Omega) \text{ and }
\divs  \mbeta(\bld u_h,\muhat)=0.
 \end{alignat*}
% \end{subequations}
\end{proposition}

\subsubsection{The stabilization operators}
\label{subsec:ns-stabilization}
\ber{For the convective stabilization operator, we take the choice leading to the classic upwinding:}
\begin{subequations}
\begin{alignat}{2}
\label{stabilization-n}
 \alpha_c (\mwwhat) := \max\{\bld{\beta}\cdot \n, 0\}\,\mwwhat \quad \forall\; \mwwhat \in L^2(\dK)^d,
\end{alignat}
where $\bld{\beta}=\mbeta(\bld{u}_h,\muhat)$ is given in \eqref{post-process-defn}.
For the viscous stabilization operator, we take 
\begin{alignat}{2}
\label{stabilization-s}
\alpha_v (\mwwhat) := \frac{\nu}{h_K} P_{\bld{M_S}}(\mwwhat) \quad \forall\; \mwwhat \in L^2(\dK)^d,
\end{alignat}
\end{subequations}
where $P_{\bld{M_S}}$ is the projection onto the space $\bld{M_S}(\dK)$, whose
$i$-th component is taken to be $M_S^\mathrm{D}(\dK)$.

\subsection{Existence, uniqueness and boundedness}
\label{subsec:ns-discrete-h1}
Now that we have completed the definition of the methods, we must ask ourselves if the approximate solutions actually exist and are unique.
The next result show that this is the case under a standard {\em smallness} condition on the data. 

\begin{theorem}[Existence, uniqueness and boundedness]
\label{thm:discrete-h1-ns}
If $\nu^{-2}\|\boldsymbol{f}\|_\Omega$ is small enough, then the HDG method \eqref{ns-HDG-equations} has a unique solution.
Furthermore, for the component 
$(\bld{u}_h,\muhat)\in \Vh\times \bld{M}_h(0)$ of the approximate solution 
the following stability bound is satisfied:
\begin{eqnarray*}
\vertiii{(\boldsymbol{u}_h,\widehat{\boldsymbol{u}}_h)}_{1,\Oh}\le C\nu^{-1}\,\|\boldsymbol{f}\|_\Omega,
\end{eqnarray*}
for a constant $C$ that depends only on the finite element 
spaces, the shape-regularity of the mesh, and the domain.
\end{theorem}

\subsection{Convergence properties}
\label{subsec:ns-error}
Having shown that the approximate solutions are well defined, we next measure how well they approximate the exact solution  by
comparing them with suitably chosen projections of the exact solution. 

\subsubsection{Projections of the errors}
\label{subsubsec:projs_ns}
Let us define the projections we are going to use in our a priori error analysis.
We denote $P_\GG$, $P_\VV$, $\Piq$, $P_{\bld{M}}$ to be the $L^2$-projections onto 
$\GG_h$, $\Vh$, $\Qh$, and $\bld{M}_h$. We also define the projection $\Piww$ into the space $\Vh$ as follows. On the element
$K$, $\Piww \bld u\in \bld{V}(K)$ is defined as follows:
\begin{subequations}
 \label{bu-projection}
\begin{alignat}{2}
 \label{bu-projection-1}
 (\Piww \bld u, \vv)_K = &\; (\bld u,\vv)_K &&\;\;\forall\; w\in \divv \GG,\\
 \label{bu-projection-2}
 \bintK{\Piww \bld u}{\mwwhat} = &\; \bintK{\bld u}{\mwwhat} &&\;\;\forall\; \mwwhat\in \bld {M_S}.
\end{alignat}
\end{subequations}
Our strategy is to first estimate the size of the projection of the errors
\begin{alignat*}{3}
 \egg = \Pigg \ml - \ml_h,\;\;&  \euu = \Piww \bld u - \bld u_h, &&\;\;
 \epp = \Piq p - p_h, && \;\;\euuhat = \Pimm \bld u - \muhat,
 \end{alignat*}
and then use the triangle inequality to estimate the size of the actual errors. 
To do that, we need to use the well-known approximation properties 
of the various $L^2$-projections. We also need the approximation properties of the 
projection $\Piww$  which we show depend on the $L^2$-projection $P_\VV$.
The following result, \ber{proven in Appendix \ref{sec:B},} is a direct consequence of 
the assumption on the stabilization space $\bld{M_S}$.
\begin{proposition}
\label{lemma:projection-b}
 For the projection $\Piww \bld u\in \VV(K)$ defined above, we have 
% \begin{subequations}
 %\label{projection-approx-b}
   \begin{alignat*}{2}
 %  \label{projection-approx-1-b}
    \|\Piww \bld u - \bld u\|_K\le &\;C\, \left(\|P_{\VV} \bld u - \bld u\|_K+h_K^{1/2}\|P_{\VV} \bld u -\bld  u\|_\dK \right)\\
%\label{projection-approx-2-b}
\|\Piww \bld u\|_{\infty,K}\le &\;C\, \|\bld u\|_{\infty,K},
 \end{alignat*}
%  \end{subequations}
 where the constant $C$ only depends on the spaces $\VV(K)$ and $\bld{M_S}(K)$.
\end{proposition}
% Here the $i$-th component of the projection $\Piww u$, $(\Piww u)_i$,
% is defined as $\Pi_{\WWD}(\bld{u}_i)$ with $\Pi_{\WWD}$ being defined in 
% \eqref{u-projection}, and $\bld{u}_i$ being the $i$-th component of the vector function $\bld{u}$.

\subsubsection{A priori error estimates}
Next, we state our main convergence result.

\begin{theorem}
\label{thm:ns-error1}
Let $(\ml_h,\bld u_h,p_h, \muhat)\in \GG_h\times \Vh\times \Qh\times \bld{M}_h(0)$ be the numerical solution of 
\eqref{ns-HDG-equations}.  Assume that
\begin{alignat*}{2}
\pol_k(K)^{d\times d}\times \pol_k(K)^d\times \pol_k(K)&\subset 
\GG(K)\times \VV(K)\times Q(K)&& \qquad\forall\;K\in \Oh,
\\
 \pol_k(F)^d&\subset \bld M(F)&&\qquad\forall\;F\in \Eh.
 \end{alignat*}
 Then, for $\nu^{-2}\|\bld{f}\|_{\Omega}$ and $\nu^{-1}\|\bld{u}\|_{\infty,\Omega}$ sufficiently small, we have
 \begin{align}
 \label{est-1}
 \|\egg\|_{\Oh} + \|e ^p\|_{\Oh}
 +\vertiii{(\euu,\euuhat)}_{1,\Oh}
 +h^{-1}\,\vertiii{(\euu,\euuhat)}_{\mbox{{\rm \tiny PF}},\Oh}
 +\|\eu\|_\Oh \le &\;C\, h^{k+1}
\,
\Xi,
 \end{align}
 where 
 $
  \Xi:=\|\ml\|_{k+1}+
\nu^{-1}\,  \|\bld{\beta}\|_{\infty,\Omega}\,\|\bld u\|_{k+1}+
\nu^{-1}\,\|p\|_{k+1}
 $ 
and the constant $C$ only depends on the finite element 
spaces, the shape-regularity of the mesh, and the domain $\Omega$.

Moreover, if 
$
\nu^{-1}\|\nabla \bld{u}\|_{\Omega}$ is small enough, $\bld{u}\in \boldsymbol{W}^{1,\infty}(\Omega)$ and 
the regularity estimate in  {\rm \cite[(2.3) ]{CesmeliogluCockburnQiu17}} holds,  then
\begin{align}
%\label{main_result_dual}
\label{est-2}
\|e_{u}\|_{\Omega}\leq C\, h^{k+2} \quad \forall k\geq 1.
\end{align}
Finally, if  $\bld{u}_h^{*}\in H(\rm div,\Omega)$ is the post-processed approximate velocity 
introduced in {\rm \cite [(2.9)]{CockburnFuQiu16}}, then we have 
$\nabla \cdot \bld{u}_h^{*}=0$ in $\Omega$, and 
\begin{align}
\label{est-3}
%\label{main_result_superconvergence}
\|\bld{u}_h^{*}-\bld{u}\|_{\Omega}\leq  C\,h^{k+2}\quad \forall k\geq 1.
\end{align}
\end{theorem}
Note that this result gives optimal convergence of the velocity gradient $\mathrm{L}_h$,
 the velocity $\bld{u}_h$ and the pressure $p_h$ approximations. It also gives two superconvergence results. The first is the one of the projections of the error in the velocity, which are of order $k+1$ for  $\vertiii{(\euu,\euuhat)}_{1,\Oh}$ and of order  $k+2$ for $\vertiii{(\euu,\euuhat)}_{\mbox{{\rm \tiny PF}},\Oh}$. The second is also for the projection of the error in the velocity. The only difference is that the 
 first superconvergence estimate does not say anything about the convergence properties of the local averages, whereas the second does. Moreover, the second superconvergence result allows the local postprocessing of the velocity $\bld{u}_h^{*}$ to be an $\bld{H}(div)$, globally divergence-free approximation to the velocity converging faster than the original approximation {$\bld{u}_h$}.

\section{\ber{Proofs} of the results of Section 4}
\label{sec:proof-ns}
In this Section, we prove  
Theorem \ref{thm:discrete-h1-ns} \ber{on the existence, uniqueness and boundedness of the approximate solution, and} the convergence properties of 
Theorem \ref{thm:ns-error1}.  

% \ber{These proofs
% can be considered as a word-by-word \ber{``translation"} of the proofs proposed in \cite{CesmeliogluCockburnQiu17} 
% for a specific  HDG method on simplicial meshes. 
% For this reason, we omit proofs of the superconvergence result in Theorem 4.4
% .
% }
\ber{We would like to emphasize that, due to the existence of the discrete-$H^1$ stability results in Theorem 2.3,  
the proofs in this section can be considered as a word-by-word ''translation`` of the corresponding
proofs in \cite{CesmeliogluCockburnQiu17}, where, for the first time, a superconvergent HDG method was analyzed for 
the incompressible Navier-Stokes equations}

To simplify the notation, we write
$
A\lesssim B 
$
to indicate that $A\le C\, B$ with a constant $C$ that only depends on the 
finite element spaces, the shape-regularity of the mesh and the domain. 

\subsection{Preliminaries}
\subsection*{\ber{Rewriting the method in a compact form}}
To facilitate the analysis, we rewrite the formulation of the methods under consideration by using the bilinear form associated to the 
 Stokes system,
\begin{subequations}
 \label{ns-forms}
 \begin{alignat}{2}
  B_h(\ml,\bld u,p, \bld{\widehat{u}};\; \mg, \bld{v}, q, \mwwhat)
  := &\; 
  \nu(\ml, \mg)_\Oh \ber{+ \nu(\bld u,\divv \mg)_\Oh - \bintEh{\muuhat}{\nu\,\mg\,\n}}\nonumber
  \\
  & \;+(\nu\,\ml,\gradv \vv)_\Oh + \bintEh{-\nu\,\ml\, \n 
  +\alpha_v(\bld u-\muuhat)}{\vv-\mwwhat}\nonumber\\
  & \;-(p,\divs \vv)_\Oh + \bintEh{p\,\n}{\vv-\mwwhat}\nonumber\\
&\;  -(\bld u, \grads q)_\Oh + \bintEh{\muuhat\cdot\n}{q},
\intertext{and the bilinear form associated to the convection,}
  \OO(\bld{\beta}; (u,\uuhat), ( w, \wwhat))
  := &\; -(\bld u\otimes \bld{\beta}, \gradv \vv)_\Oh \nonumber\\
  &\; + \bintEh{\ber{(\bld{\beta}\cdot\n)}\;\muuhat+
  \alpha_c(\bld u-\muuhat)}{\vv-\mwwhat},
 \end{alignat}
\end{subequations}
where $(\ml,\bld u,p, \muuhat)$ and $(\mg, \bld{v}, q,\mwwhat)$ lie in the space 
$\left({H^1}(\Oh)^{d\times d}+\GG_h\right) \times H^1(\Oh)^d\times H^1(\Oh)\times L^2(\Eh; 0)^d$, and 
{\color{blue}$\bld{\beta}\in \VV_{\!\!\beta}\cap \VV^*_h$} { where
%\begin{subequations}
% \label{space-beta}
 \begin{alignat*}{3}
% \label{space-beta-1}
  \VV_{\!\!\beta}:=&\;\{\vv\in{H}(\mathrm{div},\Omega):&&\;\divs \vv = 0, \vv\cdot\n|_{\dK}\in L^2(\dK),\; K\in\Oh\},\\
 %\label{space-beta-2}
  \VV^*_h:=&\;\{\vv\in{L}^2(\Oh)^d:&&\;\vv|_K\in \VV^*(K),\; K\in\Oh\}.
 \end{alignat*}
%\end{subequations}
}
% \begin{align}
% \label{beta-ns}
%  \bld{\beta} \in \{\vv\in L^2(\Omega)^d:\;\divs \vv = 0,\;\; \vv\cdot\n|_\dK\in L^2(\dK)\}.
% \end{align}
Now, the equations defining the HDG method \eqref{ns-HDG-equations} can be recast as
\begin{align}
 \label{ns-form-equation}
 B_h(\ml_h,\bld u_h,p_h, \bld{\widehat{u}}_h;\; \mg,\vv,q,\mwhat) + 
 \OO(\bld\beta; (\bld u_h, \muhat), (\vv,\mwhat)) & = (\bld f,\vv)_\Oh,
\end{align}
with $\bld{\beta}= \mbeta(\bld{u}_h,\muhat)$ defined in \eqref{post-process-defn}.
Consistency of the HDG method \eqref{ns-HDG-equations} implies that,
for the exact solution $(\ml,\bld u, p)\in {H^1}(\Omega)^{d\times d}\times H^{2}(\Omega)^d
\times H^1(\Omega)$ of \eqref{ns-equation} 
(assuming $H^2$-regularity),
\begin{align}
 \label{ns-form-equation-ex}
 B_h(\ml,\bld u,p, \bld{{u}};\; \mg,\vv,q,\mwhat) + 
 \OO(\bld u; (\bld u, \bld u), (\vv,\mwhat)) & = (\bld f,\vv)_\Oh
\end{align}
for all $(\mg,\vv,q,\mwhat)\in\GG_h\times \Vh\times \Qh\times \bld{M}_h(0)$.

\subsection*{An inequality for the viscous energy}
% Due to (\ref{require-w}) and (\ref{ns-local-spaces}), $\Vh$ is piece-wise polynomial space.
% Then, due to the discrete Sobolev embedding result in \cite[Theorem 2.1]{DiPietroDroniouErn10}, 
% we have the following discrete Poinc\'are inequality (\ref{discrete-poincare}).
% \begin{lemma}[Discrete Poinc\'are inequality]
% % [Discrete Sobolev embeddings]
% % For all $q$ such that 
% % \begin{itemize}
% %  \item [(i)] $1\le q\le \frac{2d}{d-2}$ if $d\ge 3$.
% %  \item [(ii)] $1\le q< +\infty$ if $d=2$
% % \end{itemize}
% \label{lemma:poincare}
% There exists a constant $C_{p}$, depending only on the domain $\Omega$ the shape-regularity of the mesh $\Oh$ 
% and the finite element space $\Wh$, such that
% \begin{align}
%  \label{discrete-poincare}
%  \|\vv\|_{L^{2}(\Omega)} \leq C_{p} \vertiii{(\vv,\mwhat)}_{1,\Oh}.
% \end{align}
% for all $(\vv,\mwhat)\in \Vh\times \bld{M}_h(0)$.
% \end{lemma}

Next, we obtain a key inequality for the viscous energy associated the discrete Stokes operator associated with the HDG method \eqref{ns-HDG-equations},
namely,
\begin{align}
\label{ns-energy}
\mathsf{E}(\ml,\bld u, \bld{\widehat{u}}): = &\;
B_h(\ml,\bld u,p, \bld{\widehat{u}};\; \ml,\bld u,p, \bld{\widehat{u}})\nonumber\\
=&\;
\nu(\ml, \ml)_\Oh +
\bintEh{\frac{\nu}{h_K}P_{\bld{M_S}}(\bld u-\muuhat)}{P_{\bld{M_S}}(\bld u-\muuhat)}.
\end{align}

\begin{lemma}
 \label{lemma:ns-energy}
 Let $(\ml_h,\bld u_h,p_h, {\muhat})\in \GG_h\times\Vh\times \Qh\times\bld{M}_h(0)$ 
 be the numerical solution of the linear system
 \eqref{ns-HDG-equations} with a prescribed velocity $\bld{\beta}\in \VV_{\!\!\beta}$,
 then, we have 
\[
  \mathsf{E}(\ml_h,\bld u_h, \muhat) 
  \le  (f, u_h)_\Oh.
  \]
\end{lemma}
\begin{proof} By equation \eqref{ns-form-equation} with $(\mg,\bld v,q, \mwhat):=(\ml_h,\bld u_h,p_h, \muhat)$, we get
the energy identity
\[
  \mathsf{E}(\ml_h,\bld u_h, \muhat) + \OO(\bld{\beta};(\bld u_h,\muhat),(\bld u_h,\muhat)) 
   = (\bld f,\vv)_\Oh,
 \]
and since 
 $
  \OO(\bld{\beta};(\bld u_h,\muhat),(\bld u_h,\muhat)) = \;
  \frac{1}{2}\bintEh{|\bld{\beta}\cdot\n|(\bld u_h-\muhat)}{\bld u_h-\muhat}\ge 0,
  $
  the  inequality follows. This completes the proof.
\end{proof}

\subsection*{The new discrete inequalities} Next, we relate the  viscous energy of the discrete Stokes operator with our new discrete inequalitites of 
Theorem \ref{thm:dh1}.

\begin{theorem}[Global, discrete $\bld{H}^1$- and Poincar\'e-Friedrichs inequalities]
 \label{thm:dh1-global}
Let $(\mr_h,\bld z_h,{\mzhat})\in \GG_h\times \Wh\times \Mh$ satisfy 
\vskip-.5truecm  
\[
  (\mr_h, \mg)_\Oh - (\bld z_h, \divv\mg)_\Oh + \bintEh{\mzhat}{\mg\,\n} = 0\qquad \forall\;\mg\in \GG_h.
 \]
% \vskip-.5truecm
Then,
\vskip-.5truecm
 \begin{alignat*}{2}
  \vertiii{(\bld{z}_h,\widehat{\bld{z}}_h)}_{1,\Oh}^2&\le C\,\Theta_h &&\quad\mbox{{\rm ($\bld{H}^1$)},}
  \\
 h^{-2}\, \vertiii{(\bld{z}_h,\widehat{\bld{z}}_h)}_{\mbox{{\rm \tiny PF}},\Oh}^2&\le C\,\Theta_h&&\quad\mbox{{\rm (Poincar\'e-Friedrichs)}},
 \end{alignat*}
\vskip-.5truecm
where
\vskip-.5truecm 
\[
 \Theta_h:=\sum_{K\in\Oh} ( \|\mr_h\|_{K}^2 + 
  h_K^{-1}\|P_{\bld{M}_S}(\bld{z}_h-\widehat{\bld{z}}_h)\|_\dK^2)= \nu^{-1}\, \mathsf{E}(\mr_h,\bld z_h, \widehat{\bld{z}}_h).
 \]
\noindent Here,  the constant $C$ only depends on the finite element spaces $\VV(K)$, $\WW(K)$ and $M_S(\dK)$, 
and on the shape-regularity properties of the elements  $K\in\Oh$.
\end{theorem}

\begin{proof} This result follows from the local discrete inequalities of Theorem \ref{thm:dh1}.
For $i=1,\dots,d$, let  $\mathrm{\mr}_i$ denote the $i$-th row of the matrix $\mr$, and let $\vv_i$ denote the $i$-th component of the vector $\vv$. Then,
by the choice of the local spaces \eqref{ns-local-spaces-a}, we have that, on the element $K$,
\[
((\mr_h)_i, (\bld z_h)_i, (\mzhat)_i) \in 
\VVD\times \WWD\times \MD,
\]
and since $\VVD\times \WWD$ admits an $\MD$-decomposition, 
we can apply Theorem \ref{thm:dh1} with $\mathrm{c}=\mathrm{Id}$  and
$(\bld{q}_h,u_h,\uhat):=((\mr_h)_i, (\bld z_h)_i, (\mzhat)_i)$.  The inequalities now follow by adding over all element $K\in\Oh$ and then
over the components $i=1,\dots,d$. This completes the proof.
\end{proof}

\subsection*{Properties of the convective form $\OO$}
In the next result, we gather some properties of the convective form $\OO$.
\begin{lemma}[Properties of the nonlinear term $\OO$ {\rm \cite[Proposition 3.4, Proposition 3.5]{CesmeliogluCockburnQiu17}}]
\label{lemma:OO} 
For any $(\bld{v}_h,\bld{\widehat{v}}_h)\in \Vh\times \bld{M}_h(0)$, we have
\begin{subequations}
 \begin{align}
 \label{oh-lipschitz}
 |\OO(\bld{\beta};(\bld{u},\bld{\widehat{u}}), (\bld{v}_h,\mwhat))|
 \lesssim \vertiii{(\bld{\beta},\ave{\bld{\beta}})}_{1,\Oh}
 \vertiii{(\bld{u},\bld{\widehat{u}})}_{1,\Oh}
 \vertiii{(\bld{v}_h,\bld{\widehat{v}}_h)}_{1,\Oh},
\end{align}
for all $\bld{\beta}\in \Vh^*$ and 
$(\bld{u},\bld{\widehat{u}})\in \Vh\times \bld{M}_h(0)$,
 \begin{align}
 \label{oh-linf-1}
 |\OO(\bld{\beta};(\bld{u},\bld{\widehat{u}}), (\bld{v}_h,\mwhat))|
 \lesssim \|\bld{\beta}\|_{\infty,\Omega}
 \vertiii{(\bld{u},\bld{\widehat{u}})}_{0,\Oh}
 \vertiii{(\bld{v}_h,\bld{\widehat{v}}_h)}_{1,\Oh},
\end{align}
for all $\bld{\beta}\in L^\infty(\Omega)^d\cap\Vh^*$ and 
$(\bld{u},\bld{\widehat{u}})\in H^1(\Oh)^d\times L^2(\Eh,0)^d$,
and
 \begin{align}
 \label{oh-linf-2}
 |\OO(\bld{\beta};(\bld{u},\bld{\widehat{u}}), (\bld{v}_h,\mwhat))
 -\OO(\bld{\gamma};(\bld{u},\bld{\widehat{u}}), (\bld{v}_h,\mwhat))|\hspace{2cm}\nonumber\\
 \lesssim \vertiii{(\bld{\beta}-\bld{\gamma},0)}_{0,\Oh}
 \vertiii{(\bld{u},\bld{\widehat{u}})}_{\infty,\Oh}
 \vertiii{(\bld{v}_h,\bld{\widehat{v}}_h)}_{1,\Oh},
\end{align}
for all $\bld{\beta}\in H^1(\Oh)^d+\Vh^*$ and
$(\bld{u},\widehat{\bld{u}})\in L^\infty(\Oh)^d\times L^\infty(\Eh)^d$.
\end{subequations}
\end{lemma}

%To prove this lemma, we are going to use two auxiliary  results we display next.
%
% \begin{lemma}[Trace inequalities {\rm \cite[Equation 7.7]{KarakashianJureidini98} }]
% \label{lemma:trace}
% Let $K$ be a star-shaped polygon/polyhedron.
% Let $v\in H^1(K)$, then there holds
%  \[
%h_K^{1/4}\|v\|_{L^4(\dK)}\lesssim  \|v\|_{L^4(K)} +
%\|\grads v\|_{L^2(K)}.
% \]
%\end{lemma}
%This result is a direct application of a 
%stronger result \cite[Equation 7.3]{KarakashianJureidini98}).  See \cite[Lemma 5.5]{Waluga:12} for a short proof. 
%
%\begin{lemma}[Discrete Sobolev embeddings {\rm \cite[Theorem 2.1]{DiPietroDroniouErn10}}]
%\label{lemma:poincare-ns}
%For all $p$ such that 
%\begin{itemize}
% \item [{\rm (i)}] $1\le p\le \frac{2d}{d-2}$ if $d\ge 3$.
% \item [{\rm (ii)}] $1\le p< +\infty$ if $d=2$, 
%\end{itemize}
%we have
%\begin{align*}
%\|\vv\|_{L^p(\Omega)} \lesssim \vertiii{(\vv,\mwhat)}_{1,\Oh}.
%\end{align*}
%for all $(\vv,\mwhat)\in \Vh\times \bld{M}_h(0)$.
%\end{lemma}
%
%We are now ready to prove Lemma \ref{lemma:OO}.
%\begin{proof}
% See the proof of . Note that the proof of the 
% first estimate \eqref{oh-lipschitz} uses Lemma \ref{} along with Lemma \ref{lemma:poincare-ns} for 
% $p=4$. This completes the proof of Lemma \ref{lemma:OO}.
%\end{proof}

\subsection{\ber{Proof of Theorem \ref{thm:discrete-h1-ns}}}
Now we are ready to prove 
the existence and uniqueness  of the approximation in Theorem \ref{thm:discrete-h1-ns}.
The proof is almost identical to that in \cite[Section 5]{CesmeliogluCockburnQiu17}. 

We use a Banach fixed-point theorem by constructing a
contraction mapping $\mathcal{F}: Z_h\rightarrow Z_h$, where 
\[
Z_h:=\{(\vv,\mwhat)\in \Vh\times \bld M_h(0):
(\bld{v}_h,\grads q)_\Oh-\bintEh{\mwhat\cdot\n}{q}=0\;\;\forall\; q\in \Qh\}.
\]

Let us show that \ber{there} is a ball $K_h$ inside $Z_h$ such that \ber{$\mathcal{F}$} maps $K_h$ into $K_h$.
For a pair $(\bld w_h,\widehat{\bld{w}}_h)\in Z_h$, the mapping is defined by 
$\mathcal{F}(\bld w_h,\widehat{\bld{w}}_h):=(\bld{u}_h,\muhat)$ 
with $(\bld{u}_h,\muhat)$ being part of the numerical solution to the linear system \eqref{ns-HDG-equations} with 
$\bld{\beta}=\mbeta(\bld w_h,\widehat{\bld{w}}_h)$.
By Lemma \ref{lemma:div-proj},  we have that $\bld{\beta}\in \VV_{\!\!\beta}$.
% is a divergence-free vector filed and belongs to the space in 
% \eqref{beta-ns}. 
Then, 
\begin{alignat*}{2}
\vertiii{(\bld{u}_h,\muhat)}_{1,\Oh}^2 
&\lesssim\, \nu^{-1}  \mathsf{E}(\ml_h,\bld u_h, \muhat) 
&&\quad\mbox{ by Theorem \ref{thm:dh1-global},}
\\
&\lesssim\,  \nu^{-1}(\bld{f}, \bld{u}_h)_\Oh
&&\quad\mbox{ by Lemma \ref{lemma:ns-energy},}
\\
&\lesssim\, \nu^{-1} \|\bld{f}\|_\Oh\, \|\bld{u}_h\|_\Oh
\\
&\lesssim\,  \nu^{-1}\|\bld{f}\|_\Oh\,\vertiii{(\bld{u}_h,\muhat)}_{1,\Oh},
\end{alignat*}
and we get 
\[
 \vertiii{(\bld{u}_h,\muhat)}_{1,\Oh}\lesssim \nu^{-1}\|\bld f\|_\Oh.
\]
Then, defining 
\newcommand{\cs}{C_{\mathrm{sm}}}
\[
 K_h:=\{(\vv,\mwhat)\in Z_h:\;\;\vertiii{(\vv,\mwhat)}_{1,\Oh}\le \cs\,\nu^{-1}\|\bld f\|_\Oh\},
\]
with a positive constant $\cs$ {\color{blue} big} enough, we conclude that $\mathcal{F}$ maps $K_h$ into itself.

Now we only have to show that $\mathcal{F}$ is a contraction in $K_h$. 
Set $(\bld{u}_h^1,\muhat^1):=\mathcal{F}(\bld{w}_h^1,\widehat{\bld{w}}_h^1)$ and 
$(\bld{u}_h^2,\muhat^2):=\mathcal{F}(\bld{w}_h^2,\widehat{\bld{w}}_h^2)$ with
$(\bld{w}_h^i,\widehat{\bld{w}}_h^i)\in K_h$ for $i=1,2$.
Now, let $(\ml_h^i, \bld{u}_h^i, p_h^i, \muhat^i)$ be the solution to \eqref{ns-HDG-equations} with 
$\bld{\beta}^i: = \mbeta(\bld w_h,\widehat{\bld w}_h)$. Using  $\ddl:=\ml_h^1-\ml_h^2$ and similar definitions for 
$\ddu, \ddp,\dduhat, \dbeta, \ddw$, and $\ddwhat$, and the fact that 
equation \eqref{ns-form-equation} is satisfied for $i=1,2$,
to conclude that 
\begin{alignat*}{2} 
 \mathsf{E}(\ddl, \ddu, \dduhat) = &\;
 -\OO(\bld{\beta}^1;(\bld{u}_h^1, \muhat^1),(\ddu,\dduhat))
 +\OO(\bld{\beta}^2;(\bld{u}_h^2, \muhat^2),(\ddu,\dduhat))
 \\
 = &\;
 -\OO(\dbeta,;(\bld{u}_h^1, \muhat^1),(\ddu,\dduhat))
 -\OO(\bld{\beta}^2;(\ddu, \dduhat),(\ddu,\dduhat))
\\
 \le  &\;
 -\OO(\dbeta;(\bld{u}_h^1, \muhat^1),(\ddu,\dduhat)).
\end{alignat*}
By Lemma \ref{lemma:OO}, we easily get that
\begin{alignat*}{2} 
\mathsf{E}(\ddl, \ddu, \dduhat)  \lesssim &\;
 \vertiii{(\dbeta, \ave{\dbeta})}_{1,\Oh}
 \vertiii{(\bld{u}_h^1, \muhat^1)}_{1,\Oh}
 \vertiii{(\ddu,\dduhat)}_{1,\Oh}\\
 \lesssim &\;
 \vertiii{(\ddw, \ddwhat)}_{1,\Oh}
 \vertiii{(\bld{u}_h^1, \muhat^1)}_{1,\Oh}
   \vertiii{(\ddu,\dduhat)}_{1,\Oh}
   &&\mbox{ by Proposition \ref{lemma:div-proj},}
\\
\lesssim&\;\nu^{-1}  \|\bld f\|_\Oh \vertiii{(\ddw, \ddwhat)}_{1,\Oh}
 \vertiii{(\ddu,\dduhat)}_{1,\Oh},
\end{alignat*}
by Theorem \ref{thm:discrete-h1-ns}. Combining this result with Theorem  \ref{thm:dh1-global}, we immediately get 
\[
 \vertiii{(\ddu,\dduhat)}_{1,\Oh} \lesssim\;\nu^{-2}  \|\bld f\|_\Oh \vertiii{\ber{(\ddw, \ddwhat)}}_{1,\Oh}.
\]
Hence, for $\nu^{-2}  \|\bld f\|_\Oh$ sufficiently small, the mapping $\mathcal{F}$ is a contraction
in $K_h$.  This completes the proof of Theorem  \ref{thm:discrete-h1-ns}.

\subsection{Proof of estimate \eqref{est-1} in Theorem 4.4}
% Here, we state and prove a result from which \ber{the energy estimate \eqref{est-1} in} Theorem \ref{thm:ns-error1} immediately follows. 
\ber{The energy estimate \eqref{est-1} in Theorem 4.4 directly follows from
Proposition \ref{lemma:projection-b}, the approximation properties of the finite element spaces and from
Theorem 5.4 below.
To simplify the notation, we introduce the following approximation errors:} 
\begin{alignat*}{3}
\dgg := \ml -\Pigg \ml, \;\;& \duu := \bld u - \Piww \bld u, &&\;\;
 \dpp := p -\Piq p, && \;\;\duuhat :=\bld u -  \Pimm \bld u.
\end{alignat*}
\begin{theorem}
\label{thm:ns-error}
Under the assumptions of Theorem \ref{thm:ns-error1},
we have
%\begin{subequations}
\begin{alignat*}{2}
%\label{error-u-l2-ns}
%\|\euu\|_{\Oh} \le &\;C\,
%\vertiii{(\euu,\euuhat)}_{1,\Oh},
%\\
%\label{error-u-h1-ns}
%\vertiii{(\euu,\euuhat)}_{1,\Oh}
%\le &\;
%C\,\Big(\|\egg\|_{\Oh}^2 +
%\sum_{K\in\Oh}\frac{1}{h_K}\|P_{\bld{M_S}}(\euu-\euuhat)\|_\dK^2\Big)^{1/2},\\
%\label{error-q-ns}
%\Big(\|\egg\|_{\Oh}^2 &+
%\sum_{K\in\Oh}\frac{1}{h_K}\|P_{M_S}(\euu-\euuhat)\|_\dK^2\Big)^{1/2}
 \|\egg\|_{\Oh} 
 +\vertiii{(\euu,\euuhat)}_{1,\Oh}
 +h^{-1}\,\vertiii{(\euu,\euuhat)}_{0,\Oh}
 +\|\eu\|_\Oh \le
 C\,\nu^{-1}\,\Theta_{ns}^{1/2},
\end{alignat*}
%\end{subequations}
where
\begin{align*}
%\label{theta-1-ns}
 \Theta_{ns} := &\;
\sum_{K\in \Oh}{h_K}\,(\|\nu\,\dgg\,\n\|_\dK^2 + \|\dpp\|_\dK^2)+\|\bld{u}\|_{\infty,\Omega}^2\,\vertiii{(\duu,\duuhat)}_{0,\Oh}^2.
\end{align*}
Here, the constant $C$ depends only on the finite element 
spaces, the shape-regularity of the mesh $\Oh$, and the domain $\Omega$.
\end{theorem}

\ber{The rest of this subsection is devoted to the proof of Theorem 5.4.
We need the following two auxiliary results.}
% \subsection*{Some auxiliary results}
% To \ber{prove the above theorem}, we use the following auxiliary results.
\begin{lemma}
 \label{lemma:energy-ns}
 We have 
 \begin{align*}
   B_h(\egg,\euu,\epp, \euuhat;\,\mg,\vv, q,\mwhat)
  = &\;
\bintEh{\nu\,\dgg\,\n-\dpp\,\n}{\vv-\mwhat}  \\
&\; +   \OO(\mbeta(\bld{u}_h,\muhat); (\bld u_h,\muhat),\,(\vv,\mwhat))\nonumber\\
&\; -   \OO(\bld{u}; (\bld u,\bld{u}),\,(\vv,\mwhat))\nonumber
 \end{align*}
 for all 
 $(\vv,\ww_h,\what)\in \Vh\times \Wh\times \Mh(0)$.
 \end{lemma}
\begin{proof}
It is a direct consequence of the definition of the numerical method \eqref{ns-form-equation}, the consistency of the method \eqref{ns-form-equation-ex},
and the definition of the projections in Subsection \ref{subsubsec:projs_ns}.
\ber{In particular,} note that, by the definition of $\Piww$, \ber{there holds}
\[
 \bintK{\frac{\nu}{h_K} P_{\bld{M_S}}(\duu-\duuhat)}{\vv-\mwhat} =
\frac{\nu}{h_K}  \bintK{\duu-\duuhat}{P_{\bld{M_S}}(\vv-\mwhat)} = 0.
\]
\end{proof}

% The proof of the following lemma can be found \cite[Section 6]{CesmeliogluCockburnQiu17}.
% \ber{We give the proof in Appendix \ref{sec:C} for completeness.}
\begin{lemma}
 \label{lemma:oo-term}
 We have 
 \begin{align*}
%  \label{ns-oo-error}
\OO(\mbeta(\bld{u}_h,\muhat); (\bld u_h,\muhat),\,(\euu,\euuhat))  
\\
 -   \OO(\bld{u}; (\bld u,\bld{u}),\,(\euu, \euuhat))&\;\lesssim
\|\bld{u}\|_{\infty,\Omega}\,\Phi \,\vertiii{(\euu,\euuhat)}_{1,\Oh}.
 \end{align*}
 where 
 \[
  \Phi:= \vertiii{(\euu,\euuhat)}_{0,\Oh} + \vertiii{(\duu,\duuhat)}_{0,\Oh}
  + \vertiii{(\mbeta(\Piww\bld{u}, \Pimm\bld{u})-\bld u, 0)}_{0,\Oh}.
 \]
 \end{lemma}
 \ber{
 For a proof, see Appendix \ref{sec:C}; see also \cite[Section 6]{CesmeliogluCockburnQiu17}.}
 % for a similar proof.
% \ber{We give the proof in Appendix \ref{sec:C} for completeness.}
% \end{proof}}

%\subsection*{Proof of the error estimates}
\ber{We are now ready to prove Theorem \ref{thm:ns-error}. 
Since the following estimates holds}
% Since
\begin{alignat*}{2}
\|\eu\|_\Oh 
&\le  \vertiii{(\euu,\euuhat)}_{1,\Oh},&&\mbox{ by  \cite[Theorem 2.1]{DiPietroDroniouErn10}},
\\
\vertiii{(\euu,\euuhat)}_{1,\Oh} 
&\le C\, \nu^{-1}\, \mathsf{E}(\egg,\euu, \euuhat), 
&&\mbox{ by  Theorem \ref{thm:dh1-global}},
\\
h^{-1}\,\vertiii{(\euu,\euuhat)}_{\mbox{{\rm \tiny PF}},\Oh} 
&\le  \vertiii{(\euu,\euuhat)}_{1,\Oh},
\end{alignat*}
\ber{
the left hand side of the inequality in Theorem 5.4 is smaller than
% $$
% we only need to 
% because the seminorms \eqref{localseminorms} are equivalent, 
% we only have to 
% estimate 
\[C\nu^{-1}\mathsf{E}(\egg,\euu, \euuhat).\]
We turn to estimate the above term next using a standard energy argument.
}
\ber{To do that, we take  \[(\mg,\vv,q, \mwhat):=(\egg, \euu, \epp,\euuhat)\] in Lemma \ref{lemma:energy-ns}, to get} 
\begin{align*}
\mathsf{E}(\egg,\euu, \euuhat) =&\;  
\bintEh{\nu\,\dgg\,\n-\dpp\,\n}{\euu-\euuhat}  \\
&\; +   \OO(\mbeta(\bld{u}_h,\muhat); (\bld u_h,\muhat),\,(\euu-\euuhat))\\
&\; -   \OO(\bld{u}; (\bld u,\bld{u}),\,(\euu-\euuhat)).
\end{align*}
\ber{Then,  applying the \ber{Cauchy-Schwarz} inequality and using 
Lemma \ref{lemma:OO}, we obtain}
\begin{align*}
\mathsf{E}(\egg,\euu,\euuhat) \lesssim &\;  
  (\sum_{K\in\Oh} h_K\,\|\nu\,\dgg\,\n-\dpp\,\n\|_\dK^2)^{1/2}\,\vertiii{(\euu, \euuhat)}_{1,\Oh}\\
&\; +  \|\bld{u}\|_{\infty,\Omega}
\vertiii{(\duu,\duuhat)}_{0,\Oh}\,\vertiii{(\euu, \euuhat)}_{1,\Oh}\\
&\; +  \|\bld{u}\|_{\infty,\Omega}\,\vertiii{(\mbeta(\Piww \bld u, \Pimm \bld u))-\bld u,0)}_{0,\Oh}
\,\vertiii{(\euu, \euuhat)}_{1,\Oh}\\
&\; +  \|\bld{u}\|_{\infty,\Omega}\,\vertiii{(\euu, \euuhat)}_{0,\Oh}
\,\vertiii{(\euu, \euuhat)}_{1,\Oh}.
\end{align*}
Now, assuming $\nu^{-1} \|\bld{u}\|_{\infty,\Omega}$ sufficiently small such that 
\[
\|\bld{u}\|_{\infty,\Omega}\,\vertiii{(\euu, \euuhat)}_{0,\Oh}
\,\vertiii{(\euu, \euuhat)}_{1,\Oh}\le \frac{1}{2} \mathsf{E}(\egg,\euu,\euuhat), 
\]
we get
\begin{align*}
\mathsf{E}(\egg,\euu,\euuhat) \lesssim &\;  
   \sum_{K\in\Oh} h_K\,\|\nu\,\dgg\,\n-\dpp\,\n\|_\dK^2+  \|\bld{u}\|_{\infty,\Omega}^2\,\vertiii{(\duu,\duuhat)}_{0,\Oh}^2
\end{align*}
since we have that 
\[
 \vertiii{(\mbeta(\Piww \bld u, \Pimm \bld u))-\bld u,0)}_{0,K}\lesssim
 \vertiii{(\duu,\duuhat)}_{0,K},
\]
by the approximation properties of $\mbeta$, see  Proposition \ref{lemma:div-proj}.
This completes the proof of Theorem \ref{thm:ns-error}. 

% \section{Numerical results}
\ber{
\subsection{Proofs of estimates \eqref{est-2} and \eqref{est-3} in Theorem 4.4}
The superconvergent velocity estimates in $L^2$-norm in \eqref{est-2} and \eqref{est-3}
follow from a standard duality argument. For a detailed proof,
we refer interested reader to \cite{CesmeliogluCockburnQiu17}.
}
% \subsection{Three dimensional convection-diffusion}
% \subsection{Two dimensional Navier-Stokes}
% \subsection{Three dimensional Navier-Stokes}

\section{Concluding remarks}
\label{sec:conclude}
As we pointed out in \S 2.5, 
the application of our approach to the steady-state diffusion problem gives rise to the 
first superconvergent HDG method, namely,  the so-called SFH method proposed in 
\cite{CockburnDongGuzman08} when its non-zero stabilization is taken to be of order $1/h$. 
As shown in \cite{CockburnDongGuzman08}, the convergence properties of the SFH method remain 
unchanged when the stabilization function increases. A similar phenomenon takes place for all the 
methods considered here.

The extension of the techniques developed in this paper to other nonlinear partial differential equations constitutes the subject of ongoing work.

\

\ber{{\bf Acnowledgements}. The authors would like to thank the reviewers for their constructive comments leading to a better presentation of the material of this paper.}

 \appendix
 
% \section{Properties of the convective velocity $\mbeta(\bld u_h,\muhat)$}
\section{\ber{Proof of Proposition \ref{lemma:div-proj}}}
\label{sec:A}
\ber{Here we give a proof of Proposition \ref{lemma:div-proj} on the 
properties of the convective velocity $\mbeta(\bld u_h,\muhat)$.}

The well-posedness of the projection $\mbeta(\bld u_h,\muhat)\in \VV^*(K)$ is due to properties \eqref{div-space}
 on the space $\VV^*(K)$ since we have 
 $
  \gamma((\widetilde{\VV^*}(K))^\perp) = \MD
 $;  see \cite[Proposition 6.4]{CockburnFuSayas16}.
 Then, the first two estimates  directly follows from scaling and norm-equivalence on finite dimensional spaces.

Now, \ber{assume} $(\bld{u}_h,\muhat)$ satisfies \eqref{ns-HDG-equations-3} for all 
$q\in \Qh$. 
By equation \eqref{post-process-defn-2} and property \eqref{div-space-2} on the space $\VV^*(K)$, we immediately have 
$\mbeta(\bld{u}_h,\muhat)\in H(\mathrm{div};\Omega)$. 

\ber{Let us now prove that it is divergence-free. Obviously,} \eqref{ns-HDG-equations-3} is satisfied for the constant test function $q = 1$.
Hence, we have, on each element $K$,
\[
 -(\bld u_h,\grads q)_K + \bintK{\muhat\cdot\n}{q} = 0\quad \forall\; q\in Q(K)=\WWD.
\]
Next, by the definition of $\widetilde{\VV^*}(K)$, we have $\grads Q(K)= \grads \WWD \subset \widetilde{\VV^*}(K)$. 
Hence, using the definition of $\mbeta(\bld u_h,\muhat)$ in the above 
equation, and integrating by parts, we get
\[
(\divs \mbeta(\bld u_h,\muhat),q)_K=0 \quad \forall\; q\in \WWD.
\]
This implies $\divs \mbeta(\bld u_h,\muhat) = 0$ by \eqref{div-space-1}.
This concludes the proof of Proposition \ref{lemma:div-proj}.

% \section{Approximation properties of the projection $\Piw$}
\section{\ber{Proof of Proposition \ref{lemma:projection-b}}}
% Approximation properties of the projection $\Piw$}
\label{sec:B}
\ber{Here, we give a proof of Proposition \ref{lemma:projection-b}
on the approximation properties of the projection $\Piww$.} By defintion of $\Piww \bld u$,  \eqref{bu-projection}, we have that, on
the element $K$, its  $i$-th component, $(\Piww \bld u)_i$, is defined as the element of $W^{\mbox{\rm \tiny D}}(K)$ such that
\begin{subequations}\begin{alignat}{2}{}
 \label{u-projection-1}
 ((\Piww \bld u)_i, w)_K = &\; (\bld u_i,w)_K &&\;\;\forall\; w\in W^{\mbox{\rm \tiny D}}(K),\\
 \label{u-projection-2}
 \bintK{(\Piww \bld u)_i}{\widehat{w}} = &\; \bintK{\bld u_i}{\widehat{w}} &&\;\;\forall\; \widehat{w}\in {M_S(\dK)}.
\end{alignat}
\end{subequations}
Thus, if we set $\Piw \bld u_i:=(\Piww \bld u)_i$, 
to prove our result, we only have to prove a similar result for the projection $\Piw$.

By \eqref{u-projection-1}, we have $\Piw u-P_{W}u \in \Wperp(K)$.
By \eqref{u-projection-2}, we have 
 \begin{align*}
  \bintK{\Piw u-P_{W} u}{\wwhat}
  =\bintK{ u-P_{W}u}{\wwhat} \quad\forall\;\wwhat\in M_S.
 \end{align*}
Hence, 
\begin{align*}
\|\Piw u-P_W u\|_{\dK}\le \,C_{M_S}
\|P_{M_S}( \Piw u-P_W u)\|_{\dK}
\le  {\color{red}C_{M_S}\,\|u-P_W u\|_{\dK}},
\end{align*}
since the constant $C_{M_S}$ exists by condition \eqref{condition-ms}.
Then, the first estimate follows directly by scaling and norm-equivalence of 
$\|w\|_K$ and $\|w\|_{\dK}$ for functions $w\in\Wperp(K)$.

Moreover, we have 
\begin{align*}
\|P_W u -u\|_K +h_K^{1/2}\|P_W u - u\|_{\dK}\le &\;
\|P_W u\|_K +\|u\|_K 
%\\&\;
+ h_K^{1/2}\|P_W u\|_{\dK} +h_{K}^{1/2}\| u\|_{\dK}\\
\le &\;
C\, \|P_W u\|_K +\|u\|_K +h_{K}^{1/2}\| u\|_{\dK}\\
\le&\; C\, \|u\|_K +h_{K}^{1/2}\| u\|_{\dK}\\
\le &\;C\,h_K^{d/2}\|u\|_{\infty,K}.
\end{align*}
The second estimate is obtained by  
scaling, norm-equivalence of $h_K^{d/2}\|w\|_{\infty,K}$ and $\|w\|_{K}$ for the finite dimensional space $W(K)$, 
the above estimate and the first estimate \ber{of Proposition \ref{lemma:projection-b}.}
This completes the proof of \ber{Proposition} \ref{lemma:projection-b}.

% \section{Some properties of the convection form $\OO$}
\section{\ber{Proof of Lemma \ref{lemma:oo-term}}}
\label{sec:C}
\ber{Here, we prove Lemma \ref{lemma:oo-term} on the properties of the convective term $\OO$.}
The main idea is to first split the terms on the left hand side of the estimate in Lemma \ref{lemma:oo-term} 
into the sum of the following four terms 
 \begin{align*}
T_1:=&\;  \OO(\mbeta(\bld{u}_h,\muhat); (\bld u_h,\muhat),\,(\euu,\euuhat)) \\
 &\; -\OO(\mbeta(\bld{u}_h,\muhat); (\Piww \bld ,\Pimm \bld u),\,(\euu,\euuhat)),\nonumber\\
T_2:=&\; \OO(\mbeta(\bld{u}_h,\muhat); (\Piww \bld u,\Pimm \bld u),\,(\euu,\euuhat))\\
&\; -\OO(\mbeta(\Piww \bld u,\Pimm \bld u); (\Piww \bld u,\Pimm \bld u),\,(\euu,\euuhat)),\nonumber\\ 
 T_3:=&\; \OO(\mbeta(\Piww \bld u,\Pimm \bld u); (\Piww \bld u,\Pimm \bld u),\,(\euu,\euuhat))\\
&\; - \OO(\mbeta(\Piww \bld u,\Pimm \bld u); (\bld u, \bld u),\,(\euu,\euuhat)),\nonumber\\
 T_4:=&\; \OO(\mbeta(\Piww \bld u,\Pimm \bld u); (\Piww \bld u,\Pimm \bld u),\,(\euu,\euuhat))\\
&\; - \OO(\bld u; (\bld u, \bld u),\,(\euu,\euuhat)).\nonumber
\end{align*}
and then estimate each of them.

So, by \eqref{oh-linf-2}, we have that
$T_1 = - \OO(\mbeta(\bld{u}_h,\muhat); (\euu,\euuhat),\,(\euu,\euuhat))\le 0.$

For the second term, we have 
\begin{align*}
 T_2 =&\; - \OO(\mbeta(\euu,\euuhat); (\Piww \bld u,\Pimm \bld u),\,(\euu,\euuhat))\\
 \lesssim&\;
 \vertiii{(\mbeta(\euu,\euuhat),\ave{\mbeta(\euu,\euuhat)})}_{0,\Oh}
 \vertiii{(\Piww \bld u,\Pimm \bld u)}_{\infty,\Oh} 
 \vertiii{(\euu,\euuhat)}_{1,\Oh}\\ 
 \lesssim&\;
 \vertiii{(\euu,\euuhat)}_{0,\Oh}\,
  \|\bld{u}\|_{\infty,\Omega}\,
 \vertiii{(\euu,\euuhat)}_{1,\Oh},
\end{align*}
by Proposition  \ref{lemma:div-proj}, and Proposition \ref{lemma:projection-b}.
For the third term, we have, by \eqref{oh-linf-1},
\begin{align*}
 T_3 =&\; \OO(\mbeta(\Piww\bld u,\Pimm \bld u); (\duu,\duuhat),\,(\euu,\euuhat))\\
 \lesssim&\;
 \|\mbeta(\Piww\bld u,\Pimm \bld u)\|_{\infty,\Oh}
 \vertiii{(\duu,\duuhat)}_{0,\Oh} 
 \vertiii{(\euu,\euuhat)}_{1,\Oh}\\ 
 \lesssim&\;
 \|\bld{u}\|_{\infty,\Omega}\,
 \vertiii{(\duu,\duuhat)}_{0,\Oh}\,
 \vertiii{(\euu,\euuhat)}_{1,\Oh},
\end{align*}
by  Proposition \ref{lemma:div-proj}.
For the last term, we have 
\begin{align*}
 T_4 
 \lesssim&\;
 \|\left(\mbeta(\Piww \bld u,\Pimm \bld u)-\bld u, 0\right)\|_{0,\Oh}
\|\bld{u}\|_{\infty,\Omega}
 \vertiii{(\euu,\euuhat)}_{1,\Oh},
\end{align*}
by \eqref{oh-linf-2}.
This concludes the proof of Lemma \ref{lemma:oo-term}.

%%%%%%%%%

%%%%%%%%%%%%%%%%%%%%%%%%%%%%%%%%%%%%%%%%%%%%%%%%%%%%%%%%%%%%%%%%%%%%
%
%\bibliography{./all}
%\bibliography{./home/guosheng/BIB/all}
%\bibliographystyle{siam}
%
%%%%%%%%%%%%%%%%%%%%%%%%%%%%%%%%%%%%%%%%%%%%%%%%%%%%%%%%%%%%%%%%%%%%

\end{document}